\documentclass[11pt]{amsart}
\usepackage{a4wide,color,amssymb}
\usepackage[all]{xy}
\usepackage{bbm,soul}

\newcommand{\K}{\mathcal K}

\newcommand{\C}{\mathcal C}
\newcommand{\F}{\mathcal F}

\newcommand{\Fil}{\varphi}
\newcommand{\IN}{\mathbb N}
\newcommand{\IC}{\mathbb C}
\newcommand{\IT}{\mathbb T}
\newcommand{\Ra}{\Rightarrow}

\newcommand{\w}{\omega}

\newcommand{\IZ}{\mathbb Z}
\newcommand{\Zar}{\mathfrak Z}

\newcommand{\pr}{\mathrm{pr}}
\newcommand{\Zeta}{\mathfrak Z}

\newcommand{\two}{\mathbbm 2}

\newtheorem{theorem}{Theorem}[section]
\newtheorem{problem}[theorem]{Problem}
\newtheorem{proposition}[theorem]{Proposition}
\newtheorem{lemma}[theorem]{Lemma}
\newtheorem{corollary}[theorem]{Corollary}
\newtheorem{claim}[theorem]{Claim}
\newtheorem{question}[theorem]{Question}
\theoremstyle{definition}
\newtheorem{definition}[theorem]{Definition}
\newtheorem{remark}[theorem]{Remark}
\newtheorem{example}[theorem]{Example}
\newtheorem*{definition*}{Definition}

\title{Categorically closed countable semigroups}
\author{Taras Banakh and Serhii Bardyla}
\address{T.~Banakh: Ivan Franko National University of Lviv (Ukraine) and Jan Kochanowski University in Kielce (Poland)}
\email{t.o.banakh@gmail.com}
\address{S.~Bardyla: University of Vienna, Institute of Mathematics, Kurt G\"{o}del Research Center (Austria)}
\thanks{The second author was supported by the Austrian Science Fund FWF (Grant  M 2967).}
\email{sbardyla@yahoo.com}
\makeatletter
\@namedef{subjclassname@2020}{%
  \textup{2020} Mathematics Subject Classification}
\makeatother


\begin{document}
\begin{abstract}
In this paper we establish a connection between categorical closedness and nontopologizability of semigroups. In particular, for the class $\mathsf{T_{\!1}S}$ of $T_1$ topological semigroups we prove that a countable semigroup $X$ with finite-to-one shifts is injectively $\mathsf{T_{\!1}S}$-closed if and only if $X$ is $\mathsf{T_{\!1}S}$-discrete in the sense that every $T_1$ semigroup topology on $X$ is discrete. Moreover, a countable cancellative semigroup $X$ is absolutely $\mathsf {T_{\!1}S}$-closed if and only if every homomorphic image of $X$ is $\mathsf{T_{\!1}S}$-discrete. Also, we introduce and investigate a new notion of a polybounded semigroup. It is proved that a countable semigroup $X$ with finite-to-one shifts is polybounded if and only if $X$ is $\mathsf{T_{\!1}S}$-closed if and only if $X$ is $\mathsf{T_{\!z}S}$-closed, where $\mathsf{T_{\!z}S}$ is the class of Tychonoff zero-dimensional topological semigroups. We show that polybounded cancellative semigroups are groups, and  polybounded $T_1$ paratopological groups 
are topological groups.
\end{abstract}
\maketitle

\section{Introduction}

In many cases, completeness properties of various objects of General Topology and Topolo\-gical Algebra can be characterized externally as closedness in ambient objects. For example, a metric space $X$ is complete if and only if $X$ is closed in any metric space containing $X$ as a subspace. A uniform space $X$ is complete if and only if $X$ is closed in any uniform space containing $X$ as a uniform subspace. A topological group $G$ is Ra\u\i kov complete  if and only if it is closed in any topological group containing $G$ as a subgroup.


This motivates the following general definition.

\begin{definition} Let $\C$ be a class of topological semigroups. A topological semigroup $X$ is defined to be
\begin{itemize}
\item {\em $\C$-closed} if for any isomorphic topological embedding $e{:}\hbox{ }X\to Y$ into a topological semigroup $Y\in\C$ the image $e[X]$ is closed in $Y$;
\item {\em injectively $\C$-closed} if for any injective continuous homomorphism $i{:}\hbox{ }X\to Y$ to a topological semigroup $Y\in\C$ the image $i[X]$ is closed in $Y$;
\item {\em absolutely $\C$-closed} if for any continuous homomorphism $h{:}\hbox{ }X\to Y$ to a topological semigroup $Y\in\C$ the image $h[X]$ is closed in $Y$.
\end{itemize}
\end{definition}

Clearly, for an{\color{blue}y} topological semigroup the following implications hold:
$$\mbox{absolutely $\C$-closed $\Ra$ injectively $\C$-closed $\Ra$ $\C$-closed}.$$

Topological semigroups satisfying the above completeness-like properties are called {\em categorically closed}. Categorically closed topological groups were investigated in
~\cite{AC,AC1,BL,Ban,Ban1,BGR,DU,G,JS, Z1,KOO,L,Z2}. 
Corresponding notions of completeness in Category Theory were investigated in~\cite{CDT,Cbook1,Er,FG,G1,GH,LW}. In particular, closure operators in different categories were studied in~\cite{BGH, Cbook, CG1, CG2, DT, Dbook, GS, T, Za}. Categorically closed topological semilattices were investigated
in~\cite{BBm, BBR, GutikPagonRepovs2010, GutikRepovs2008, Stepp75}.
For more information about complete topological semilattices and pospaces we refer to the recent survey of the authors~\cite{BBc}.

In this paper we shall be mainly interested in the categorical closedness of discrete topological semigroups. Since each semigroup can be endowed with the discrete topology and thus identified with a discrete topological semigroup, it is natural to define (injectively or absolutely) $\C$-closed semigroups as  follows.

 \begin{definition} A semigroup $X$ is defined to be ({\em injectively or absolutely}) {\em $\C$-closed} if so is the topological semigroup $X$ endowed with the discrete topology.
\end{definition}

The results of this paper show that the categorical closedness of semigroups is related to their topologizability properties. Many such properties can be defined as follows.


\begin{definition} Let $\C$ be a class of topological semigroups.  A semigroup $X$ is called
\begin{itemize}
\item {\em $\C$-discrete} (or else {\em $\C$-nontopologizable}) if for every injective homomorphism $i:X\to Y$ to a topological semigroup $Y\in\C$, the image $i[X]$ is a discrete subspace of $Y$;
\item {\em $\C$-topologizable} if $X$ is not $\C$-discrete;
\item {\em projectively $\C$-discrete} if for every homomorphism $h:X\to Y$ to a topological semigroup $Y\in\C$, the image $h[X]$ is a discrete subspace of $Y$.
\end{itemize}
\end{definition}
If the class $\C$ is closed under taking topological subsemigroups, then a semigroup is $\C$-topologizable if and only if it is algebraically isomorphic to a non-discrete topological semigroup in the class $\C$. Topologizable and nontopologizable semigroups are actively studied in Topological Algebra since the seminal paper of Markov \cite{Mar}, see \cite{BPS, DI, Hesse, KOO, Kotov, Malcev, Shelah, Taimanov78, Tro, Ztop}.
\smallskip

For any congruence $\approx$ on a semigroup $X$, the quotient set $X/_\approx$ has a unique structure of a semigroup such that the quotient map $X\to X/_\approx$ is a homomorphism. The semigroup $X/_\approx$ is called a {\em quotient semigroup} of $X$. A subset $I$ of a semigroup $X$ is called an {\em ideal} if for any $x\in I$ and $y\in X$ the elements $xy$ and $yx$ belong to $I$.
Every ideal $I$ determines the congruence $(I\times I)\cup\{(x,y)\in X\times X:x=y\}$ on $X$. The quotient semigroup of $X$ by this congruence is denoted by $X/I$. The semigroup $X$ can be identified with the quotient semigroup $X/\emptyset$.

A quotient semigroup of a $\C$-closed semigroup is not necessarily $\C$-closed (see Example 1.8 from~\cite{BB}). This fact motivates the following definition. 

\begin{definition}
A semigroup $X$ is defined to be
\begin{itemize}
\item {\em projectively $\C$-closed} if for every congruence $\approx$ on $X$ the quotient semigroup $X/_\approx$ is $\C$-closed;
\item {\em ideally $\C$-closed} if for every ideal $I\subseteq X$ the quotient semigroup $X/I$ is $\C$-closed.
\end{itemize}
\end{definition}

For any semigroup  we have the implications:
$$
\xymatrix{
\mbox{absolutely $\C$-closed}\ar@{=>}[d]\ar@{=>}[r]&\mbox{projectively $\C$-closed}\ar@{=>}[r]&\mbox{ideally $\C$-closed}\ar@{=>}[d]\\
\mbox{injectively $\C$-closed}\ar@{=>}[rr]&&\mbox{$\C$-closed}.
}
$$



For a semigroup $X$, its
\begin{itemize}
\item  {\em $0$-extension} is the semigroup $X^0=\{0\}\cup X$ where $0\notin X$ is any element such that $0x=0=x0$ for all $x\in X^0$;
\item {\em $1$-extension} is the semigroup $X^1=\{1\}\cup X$ where $1\notin X$ is any element such that $1x=x=x1$ for all $x\in X^1$.
\end{itemize}

A topology $\tau$ on a semigroup $X$ is called a {\em semigroup topology} if the binary operation of $X$ is continuous with respect to the topology $\tau$.

\begin{definition}
A semigroup $X$ is called {\em zero-closed} if $X$ is closed in $(X^0,\tau)$ for any $T_1$ semigroup topology $\tau$ on $X^0$.
\end{definition}

For $i\in\{0,1,2,3,3\frac12\}$ we denote by $\mathsf T_{\!i}\mathsf S$ the class of topological semigroups satisfying the separation axiom $T_i$ (see~\cite[\S1.5]{Eng} for more details). The class of all zero-dimensional topological semigroups (i.e., $T_1$ topological semigroups possessing a base consisting of clopen sets) is denoted by $\mathsf{T_{\!z}S}$.  
Note that any zero-dimensional space is Tychonoff and any $T_1$ space with a unique non-isolated point is zero-dimensional.
This yields the following chain of implications holding for every semigroup:
$$
\mbox{$\mathsf{T_{\!1}S}$-closed}\Ra\mbox{$\mathsf{T_{\!2}S}$-closed}\Ra\mbox{$\mathsf{T_{\!3}S}$-closed}\Ra\mbox
{$\mathsf{T_{\!3\frac12}S}$-closed}\Ra\mbox{$\mathsf{T_zS}$-closed}\Ra\mbox{zero-closed}.$$



By a {\em semigroup polynomial} on a semigroup $X$ we understand a function $f{:}\hbox{ }X\to X$ of the form $f(x)=a_0xa_1\cdots xa_n$ for some elements $a_0,\dots,a_n\in X^1$ where the number $n\ge 1$ is called a {\em degree} of $f$. 
Note that every semigroup polynomial $f$ on a topological semigroup $X$ is continuous.

\begin{definition}\label{d:polybounded} A subset $A$ of a semigroup $X$ is called {\em polybounded in} $X$ if $$A\subseteq \bigcup_{p\in P}\bigcup_{b\in B}\{x\in X:p(x)=b\}$$for a finite set $B\subseteq X$ and a finite set $P$ of semigroup polynomials on $X$.

A semigroup $X$ is called {\em polybounded} if $X$ is polybounded in $X$.
\end{definition}

Polybounded semigroups will be studied in Section~\ref{s:poly}. In Section~\ref{s:poly-shift} we shall study polybounded semigroups with finite-to-one shifts.

\begin{definition}
A semigroup $X$ is defined to have {\em finite-to-one shifts} if for any $a,b\in X$ the sets $\{x\in X:ax=b\}$ and $\{x\in X:xa=b\}$ are finite.
\end{definition}

The class of semigroups with finite-to-one shifts includes all groups and, more generally, all cancellative semigroups. Let us recall that a semigroup $X$ is {\em cancellative} if for any $a,b\in X^1$ the two-sided shift $s_{a,b}:X\to X$, $s_{a,b}:x\mapsto axb$, is injective.

The following description of the structure of polybounded cancellative semigroups will be proved in Section~\ref{s:main}.

\begin{theorem}\label{t:group} Every nonempty polybounded cancellative semigroup is a group.
\end{theorem}

One of the main results of this paper is the following characterization of $\C$-closed countable semigroups with finite-to-one shifts.

\begin{theorem}\label{char} For every countable semigroup $X$ with finite-to-one shifts, the following conditions are equivalent:
\begin{enumerate}
\item $X$ is $\mathsf{T_{\!1}S}$-closed;
\item $X$ is $\mathsf{T_{\!z}S}$-closed;
\item $X$ is zero-closed;
\item $X$ is polybounded.
\end{enumerate}
\end{theorem}

By Proposition~\ref{qd}, a homomorphic image of a polybounded semigroup is polybounded. Since any homomorphic image of a group is a group, Theorems~\ref{t:group} and \ref{char}  imply the following characterization.

\begin{corollary}\label{c:ccs}
For a countable cancellative semigroup $X$ the following conditions are equivalent:
\begin{enumerate}
\item $X$ is polybounded;
\item $X$ is zero-closed;
\item $X$ is  $\mathsf{T_{\!1}S}$-closed;
\item $X$ is projectively $\mathsf{T_{\!1}S}$-closed.
\end{enumerate}
\end{corollary}

Corollary~\ref{c:ccs} is specific for {\em countable} cancellative semigroups and does not generalize to groups of arbitrary cardinality: by \cite{BanS}, for every infinite cardinal $\kappa$ with $\kappa^+=2^\kappa$, there exists a  simple group of cardinality $\kappa^+$ which is  absolutely $\mathsf{T_{\!1}S}$-closed but not polybounded. A group $X$ is {\em simple} if every normal subgroup of $X$ is either trivial or equals $X$.
\smallskip

On the other hand, the equivalence $(3)\Leftrightarrow(4)$ in Corollary~\ref{c:ccs} does hold for arbitrary groups and more generally for arbitrary semigroups whose all homomorphic images have finite-to-one shifts.

\begin{theorem}\label{t:C=pC} Let $X$ be a semigroup such that every homomorphic image of $X$ has finite-to-one shifts. Let $\mathsf i\in\{1,2,3,3\frac12,\mathsf z\}$. The semigroup $X$ is $\mathsf{T_{\!i}S}$-closed if and only if $X$ is projectively $\mathsf{T_{\!i}S}$-closed.
\end{theorem}

The following corollary of Theorem~\ref{t:C=pC} answers the ``group'' part of Question 9.2 in \cite{BB}.

\begin{corollary}\label{c:C=pC}  Let $\mathsf i\in\{1,2,3,3\frac12,\mathsf z\}$. A group $X$ is $\mathsf{T_{\!i}S}$-closed if and only if $X$ is projectively $\mathsf{T_{\!i}S}$-closed.
\end{corollary}


Next, we characterize the injective (and absolute) $\mathsf{T_{\!1}S}$-closedness of (cancellative) countable semigroups with finite-to-one shifts in terms of their (projective) $\mathsf{T_{\!1}S}$-discreteness.

\begin{theorem}\label{main2} A countable semigroup $X$ with finite-to-one shifts is injectively $\mathsf{T_{\!1}S}$-closed if and only if $X$ is $\mathsf{T_{\!1}S}$-discrete.
\end{theorem}

\begin{theorem}\label{tag}
 A countable cancellative semigroup $X$ is  absolutely $\mathsf{T_{\!1}S}$-closed if and only if $X$ is projectively $\mathsf{T_{\!1}S}$-discrete.
\end{theorem}

In fact, the ``only if'' parts of the characterization Theorems~\ref{main2} and \ref{tag} hold for an arbitrary semigroup. This motivates the following problem.

\begin{problem}\label{prob:d=>c} Is every $\mathsf{T_{\!1}S}$-discrete group $\mathsf{T_{\!1}S}$-closed?
\end{problem}

The answer to Problem~\ref{prob:d=>c} is affirmative for commutative groups as follows from our next characterization.



\begin{theorem}\label{t:aC-cgrp} Let $\C$ be a class of topological semigroups such that $\mathsf{T_{\!z}S}\cap\mathsf{TG}\subseteq\C\subseteq\mathsf{T_{\!1}S}$. For a commutative cancellative semigroup $X$ the following conditions are equivalent:
\begin{enumerate}
\item $X$ is injectively $\C$-closed;
\item $X$ is absolutely $\C$-closed;
\item $X$ is $\C$-discrete;
\item $X$ is finite.
\end{enumerate}
\end{theorem}

Theorem~\ref{t:aC-cgrp} is specific for commutative cancellative semigroups. For non-commutative groups the situation is totally different. In the following theorem by a {\em bounded} semigroup we understand a semigroup $X$ for which there exists $n\in\IN$ such that the $n$-th power $x^n$ of any element $x\in X$ is an idempotent. 

\begin{theorem}\label{t:KOO} Every countable bounded group without elements of order $2$ is a subgroup of an absolutely $\mathsf{T_{\!1}S}$-closed countable bounded simple group.
\end{theorem}

\begin{proof} By Theorem 2.3 in \cite{KOO}, every countable bounded group $G$ without elements of order 2 is a subgroup of a countable simple bounded group $X$, which is $\mathsf{TG}$-discrete for the class $\mathsf{TG}$ of Hausdorff topological groups. Since each semigroup topology on a bounded group is a group topology, the $\mathsf{TG}$-discrete group $X$ is $\mathsf{T_{\!1}S}$-discrete. By Theorem~\ref{main2}, the group $X$ is injectively $\mathsf{T_{\!1}S}$-closed. To show that $X$ is absolutely $\mathsf{T_{\!1}S}$-closed, take any homomorphism $h:X\to Y$ to a $T_1$ topological semigroup $Y$. Let $e$ be the unique idempotent of $X$. Since $X$ is a group, the image $h[X]$ is a group and $H=h^{-1}(h(e))$ is a normal subgroup of $X$. Since the group $X$ is simple, $H=X$ or $H=\{e\}$. If $H=X$, then the image $h[X]$ is a singleton and hence $h[X]$ is closed in $Y$. If $H=\{e\}$, then the homomorphism $h:X\to Y$ is injective and $h[X]$ is closed in $Y$ by the injective $\mathsf{T_{\!1}S}$-closedness of $X$.
\end{proof}

The following counterpart of Theorem~\ref{t:KOO} for uncountable groups is proved in \cite{BanS}.

\begin{theorem} Let $\kappa$ be a cardinal such that $\kappa^+=2^\kappa$. Every group $H$ of cardinality $|H|\le\kappa$ is a subgroup of an absolutely $\mathsf{T_{\!1}S}$-closed $36$-Shelah simple group $G$ of cardinality $|G|=\kappa^+$.
\end{theorem}

A semigroup $X$ is called {\em $n$-Shelah} if $X=\{a_1a_2\cdots a_n:a_1,\dots,a_n\in A\}$ for any subset $A\subseteq X$ of cardinality $|A|=|X|$.

\smallskip

 Theorem~\ref{t:KOO} is applied in the following example  showing that   Corollary~\ref{c:ccs} and Theorem~\ref{t:C=pC} cannot be generalized to the class of all semigroups with finite-to-one shifts.

\begin{example} There exists a countable semigroup $X$ such that
\begin{enumerate}
\item $X$ has finite-to-one shifts;
\item $X$ is $\mathsf{T_{\!1}S}$-discrete but not projectively $\mathsf{T_{\!z}S}$-discrete;
\item $X$ is injectively $\mathsf{T_{\!1}S}$-closed but not absolutely $\mathsf{T_{\!z}S}$-closed.
\end{enumerate}
\end{example}

\begin{proof} Take any countable infinite bounded commutative group $A$ without elements of order 2.  By Theorem~\ref{t:KOO}, $A$ is a subgroup of an absolutely $\mathsf{T_{\!1}S}$-closed countable group $G$. Let $\two$ be the two-element semilattice $\{0,1\}$ endowed with the operation of minimum. We claim that the subsemigroup $X=(\{0\}\times G)\cup(\{1\}\times A)$ of $\two\times G$ has the required properties. It is easy to see that the semigroup $X$ has finite-to-one shifts. To see that $X$ is $\mathsf{T_{\!1}S}$-discrete, take any injective homomorphism $h:X\to Y$ to a $T_1$ topological semigroup $Y$. By Theorem~\ref{main2}, the injective  $\mathsf{T_{\!1}S}$-closedness of the group $G$ implies that $G$ is $\mathsf{T_{\!1}S}$-discrete and hence the image $h[\{0\}\times G]$ is a discrete subspace of $Y$. To see that $h[X]$ is discrete, take any $x\in X$ and let $e$ be the unique idempotent of the subgroup $\{0\}\times G$ of $X$. Since $h[\{0\}\times G]$ is a discrete subspace of $Y$, there exists an open set $U\subseteq Y$ such that $U\cap h[\{0\}\times G]=\{h(ex)\}$. By the continuity of shifts in the topological semigroup $Y$, the set $V=\{y\in h[X]:ey\in U\}=\{y\in h[X]:ey=ex\}$ is open in $h[X]$. It is easy to see that this set contains exactly two elements, one of which is $x$.  Since $Y$ is a $T_1$-space, the singleton $V\setminus \{x\}$ is closed in $Y$ and hence $\{x\}=V\setminus(V\setminus\{x\})$ is open in $h[X]$, witnessing that the space $h[X]$ is discrete and the semigroup $X$ is $\mathsf{T_{\!1}S}$-discrete. By Theorem~\ref{main2}, the semigroup $X$ is injectively $\mathsf{T_{\!1}S}$-closed.

To see that $X$ is not absolutely $\mathsf{T_{\!z}S}$-closed, consider the homomorphism $h:X\to A^0$, defined by
$$h(i,x)=\begin{cases}
x&\mbox{if $i=1$};\\
0&\mbox{if $i=0$};
\end{cases}
$$for any $(i,x)\in X$. By Theorem~\ref{t:aC-cgrp}, the infinite abelian group $A$ is not injectively $\mathsf{T_{\!z}S}$-closed, which implies that its $0$-extension $A^0$ is not injectively $\mathsf{T_{\!z}S}$-closed and the semigroup $X$ is not absolutely $\mathsf{T_{\!z}S}$-closed.  Also Theorem~\ref{t:aC-cgrp} ensures that the group $A$ is not $\mathsf{T_{\!z}S}$-discrete, which implies that its $0$-extension $A^0$ is not $\mathsf{T_{\!z}S}$-discrete and hence $X$ is not projectively $\mathsf{T_{\!z}S}$-discrete.
\end{proof}


\section{Auxiliary results}


In this paper we denote by $\w$ the set $\{0,1,2,\dots\}$ of nonnegative integers and by $\IN$ the set $\{1,2,\dots\}$ of positive integers. For a set $X$, we denote by $|X|$ the cardinality of $X$.

The following technical lemma yields a sufficient condition of non-zero-closedness and will be used in the proof of Theorem~\ref{t:semibound}.

\begin{lemma}\label{l0} Let $X$ be a semigroup and $\mathcal K$ be a countable family of infinite subsets of $X$ such that
\begin{enumerate}
\item for any $K,L\in\mathcal K$ there exists $M\in\mathcal K$ such that $KL\subseteq M$;
\item for any $a,b\in X^1$ and $K\in\mathcal K$ there exists $L\in\mathcal K$ such that $aKb\subseteq L$;
\item for any $K\in\mathcal K$ and $a,b,c\in X^1$ the set $\{x\in K:axb=c\}$ is finite;
\item for any $K,L\in\mathcal K$ and $c\in X$ the set $\{(x,y)\in K\times L:xy=c\}$ is finite.
\end{enumerate}
Then the topology $$\tau^0=\{V\subseteq X^0:0\in V\Ra \forall K\in\mathcal K\;(|K\setminus V|<\w)\}$$turns $X^0$ into a Hausdorff topological semigroup with a unique non-isolated point $0$.
\end{lemma}

\begin{proof}
Let $\{K_n\}_{n\in\w}$ be an enumeration of the countable family $\K$. For every $n\in\w$ let $U_n=\bigcup_{i<n}K_i$. Note that $U_0=\emptyset$ and $(U_n)_{n\in\IN}$ is an increasing sequence of infinite subsets of $X$. The conditions (1)--(4) imply the conditions:
\begin{itemize}
\item[$(1')$] for any $n\in\w$ there exists $m\in\w$ such that $U_nU_n\subseteq U_m$;
\item[$(2')$] for any $a,b\in X^1$ and $n\in\w$ there exists $m\in\w$ such that $aU_nb\subseteq U_m$;
\item[$(3')$] for any $n\in\w$ and $a,b,c\in X^1$ the set $\{x\in U_n:axb=c\}$ is finite;
\item[$(4')$] for any $n\in\w$ and $c\in X$ the set $\{(x,y)\in U_n\times U_n:xy=c\}$ is finite.
\end{itemize}

Observe that the topology $\tau^0$ coincides with the Hausdorff topology
 $$\{V\subseteq X^0:0\in V\;\Ra\;\forall n\in\w\;(|U_n\setminus V|<\w)\}.$$
The definition of the topology $\tau^0$ guarantees that $0$ is the unique non-isolated point of the  topological space $(X^0,\tau^0)$. So, it remains to prove that $(X^0,\tau^0)$ is a topological semigroup.

First we show that for every $a,b\in X^1$ the shift $s_{a,b}:X^0\to X^0$, $x\mapsto axb$, is continuous. Since $0$ is a unique non-isolated point of $(X^0,\tau^0)$, it suffices to check the continuity of the shift $s_{a,b}$ at $0$. Given any neighborhood $V\in\tau^0$ of $0$, we need to show that the set $s_{a,b}^{-1}(V)=\{x\in X^0:axb\in V\}$ belongs to the topology $\tau^0$. This will follow as soon as we check that for every $n\in\w$ the set $U_n\setminus s_{a,b}^{-1}(V)$ is finite. Fix any $n\in\omega$. By condition $(2')$, there exists a number $m\in\w$ such that $aU_nb\subseteq U_m$. Then
$$U_n\setminus s_{a,b}^{-1}(V)=\{x\in U_n:axb\notin V\}=\{x\in U_n:axb\in U_m\setminus V\}.$$Since the set $V$ is open, the definition of the topology $\tau^0$ guarantees that the difference $U_m\setminus V$ is finite. Applying condition $(3')$, we conclude that the set $U_n\setminus s_{a,b}^{-1}(V)$ is finite. Therefore, the shift $s_{a,b}$ of the semigroup $X^0$ is continuous with respect to the topology $\tau^0$ and the semigroup operation is continuous on the subset $(X^0\times X)\cup (X\times X^0)$.

So, it remains to check the continuity of the semigroup operation at $(0,0)$. Fix any neighborhood $V\in\tau^0$ of $0=00$. By condition $(1')$, for every $n\in\w$ there exists $m_n\in\w$ such that $U_nU_n\subseteq U_{m_n}$. The definition of the topology $\tau^0$ ensures that for every $n\in\w$ the set $U_{m_n}\setminus V$ is finite. Applying condition $(4')$, we conclude that for every $n\in\w$ the set
$$\Pi_n=\{(x,y)\in U_n\times U_n:xy\notin V\}=\{(x,y)\in U_n\times U_n:xy\in U_{m_n}\setminus V\}$$is finite. Then we can find a finite set $P_n\subseteq U_n$ such that $\Pi_n\subseteq P_n\times P_n$.

Consider the set $$W=\{0\}\cup\bigcup_{n\in\IN}\big((U_n\setminus U_{n-1})\setminus P_n\big)$$
and observe that for every $n\in\IN$ the complement
$$U_n\setminus W\subseteq\bigcup_{1\leq k\le n}P_k$$is finite. Then $W\in\tau^0$ by the definition of the topology $\tau^0$.
We claim that $WW\subseteq V$. Assuming the opposite, find $x,y\in W$ such that $xy\notin V$. Let $n,k\in\IN$ be unique numbers such that $x\in (U_n\setminus U_{n-1})\setminus P_n$ and $y\in (U_k\setminus U_{k-1})\setminus P_k$. If $n\le k$, then $(x,y)\in\Pi_k$ and $x,y\in P_k$, which contradicts the choice of $y$. If $k\le n$, then $(x,y)\in \Pi_n$ and $x,y\in P_n$, which contradicts the choice of $x$. In both cases we obtain a contradiction, which completes the proof of the continuity of the semigroup operation on $(X^0,\tau^0)$.
\end{proof}




A nonempty family $\F$ of nonempty subsets of a set $X$ is called a {\em filter} on $X$ if $\F$ is closed under intersections and taking supersets in $X$. A subfamily $\mathcal B\subseteq\F$ is called a {\em base} of a filter $\F$ if each set $F\in\F$ contains some set $B\in\mathcal B$. In this case we say that the filter $\F$ is generated by the base $\mathcal B$.
A filter $\F$ on $X$ is
\begin{itemize}
\item {\em free} if $\bigcap\F=\emptyset$;
\item {\em principal} if $\{x\}\in\F$ for some $x\in X$;
\item an {\em ultrafilter} if for any $F\subseteq X$ either $F\in \F$ or $X\setminus F\in \F$.
\end{itemize}

The set $\Fil(X)$ of filters on a semigroup $X$ has a natural structure of a semigroup: for any filters $\mathcal E,\F$ on $X$ their product $\mathcal E\F$ is the filter generated by the base $\{EF:E\in\mathcal E,\;\;F\in\F\}$ where $EF=\{xy:x\in E,\;y\in F\}$. Identifying each element $x\in X$ with the principal filter $\{F\subseteq X:x\in F\}$, we identify the semigroup $X$ with a subsemigroup of the semigroup $\Fil(X)$.

\begin{proposition}\label{t:T1} A semigroup $X$ is $\mathsf{T_{\!1}S}$-closed if for any free ultrafilter $\F$ on $X$ there are elements $a_0,\ldots,a_n\in X^1$ such that the filter $a_0\F a_1\F\cdots \F a_n$ is neither free nor principal.
\end{proposition}

\begin{proof} To derive a contradiction, assume that a semigroup $X$ is not $\mathsf{T_{\!1}S}$-closed, but for any free ultrafilter $\F$ on $X$ there are elements $a_0,\dots,a_n\in X^1$ such that the filter $a_0\F a_1\F\cdots \F a_n$ is neither free nor principal. By our assumption, $X$ is a non-closed subsemigroup of some $T_1$ topological semigroup $Y$  whose topology is denoted by $\tau_Y$.
Take any element $y\in \overline{X}\setminus X$ and consider the filter $\mathcal H$ on $X$ generated by the family $\{X\cap U:y\in U\in\tau_Y\}$.  Since the space $Y$ is $T_1$, the filter $\mathcal H$ is free. Let $\F$ be any ultrafilter on $X$ which contains $\mathcal H$. By our assumption, there exist elements $a_0,\dots,a_n\in X^1\subset Y^1$ such that the filter $a_0\F\cdots \F a_n$ is neither free nor principal.
Since this filter is not free, there exists an element $z\in \bigcap_{F\in \F}a_0F\cdots Fa_n$. 
We claim that $a_0ya_1\cdots ya_n=z$. In the opposite case, we can find a neighborhood $U\subseteq Y$ of $y$ such that $a_0Ua_1\cdots Ua_n\subseteq Y\setminus\{z\}$. Then for the set $F=U\cap X\in\mathcal H\subseteq \mathcal{F}$ we obtain $z\notin a_0Fa_1\cdots Fa_n$, which contradicts the choice of $z$. Hence $a_0ya_1\cdots ya_n=z$. Since $X$ is a discrete subspace of $Y$, there exists an open set $V\subseteq Y$ such that $V\cap X=\{z\}$. By the continuity of the semigroup operation on $Y$, the point $y$ has a neighborhood $W\subseteq Y$ such that $a_0Wa_1\cdots Wa_n\subseteq V$. Then the set
$F=X\cap W\in\F$ has the property $a_0Fa_1\cdots Fa_n\subseteq X\cap V=\{z\}$, implying that the filter $a_0\F a_1\cdots \F a_n$ is principal. But this contradicts the choice of the points $a_0,\dots,a_n$.
\end{proof}

\begin{proposition}\label{t:iT1} A semigroup $X$ is injectively $\mathsf{T_{\!1}S}$-closed if for any free ultrafilter $\F$ on $X$ there are elements $a_0,\ldots,a_n\in X^1$ and distinct elements $u,v\in X$ such that $$\{u,v\}\subseteq \bigcap_{F\in\F}a_0F a_1F\cdots F a_n.$$
\end{proposition}

\begin{proof}
Assuming that $X$ is not injectively $\mathsf{T_{\!1}S}$-closed, we can find a continuous injective homomorphism $h:X\rightarrow Y$ into a topological semigroup $(Y,\tau_Y)\in\mathsf{T_{\!1}S}$ such that $h[X]$ is not closed in $Y$. Fix any element $y_0\in \overline{h[X]}\setminus h[X]$. Since the space $Y$ is $T_1$, the filter $\mathcal H$ on $X$ generated by the base $\{h^{-1}[U]:y_0\in U\in\tau_Y\}$ is free. Let $\F$ be any ultrafilter on $X$ which contains $\mathcal H$. By the assumption, there exist elements $a_0,\ldots,a_n\in X^1$ and distinct elements $u,v\in X$ such that $\{u,v\}\subseteq a_0F a_1\cdots Fa_n$ for any $F\in \F$. 
Since $(Y,\tau_Y)$ is a $T_1$-space, the point $h(a_0)y_0h(a_1)\cdots y_0h(a_n)\in Y$ has a neighborhood $W\in\tau_Y$ such that $|W\cap\{h(u),h(v)\}|\le 1$ and hence $|\{u,v\}\cap h^{-1}[W]|\le 1$. By the continuity of semigroup operation on $Y$, there exists a neighborhood $U\in\tau_Y$ of $y_0$ such that $h(a_0)Uh(a_1)\cdots Uh(a_n)\subseteq W$. Note that $h^{-1}[U]\in\mathcal H\subseteq\F$ and hence $$\{u,v\}\subseteq a_0h^{-1}[U]a_1\cdots h^{-1}[U]a_n=h^{-1}[h(a_0)Uh(a_1)\cdots Uh(a_n)]\subseteq h^{-1}[W],$$ which contradicts the choice of $W$.
\end{proof}

The following example constructed by Taimanov~\cite{Taimanov73} shows that there exists an injectively $\mathsf{T_{\!1}S}$-closed semigroup which is not ideally $\mathsf{T_{\!z}S}$-closed, and hence for any class $\C$ with $\mathsf{T_{\!z}S}\subseteq\C\subseteq\mathsf{T_{\!1}S}$, the injective $\C$-closedness does not imply the ideal $\C$-closedness.

\begin{example} Given an infinite cardinal $\kappa$, consider the semigroup $X=(\kappa,*)$ endowed with the binary operation
$$x*y=\begin{cases}1,&\mbox{if $x=y>1$};\\
0,&\mbox{otherwise.}
\end{cases}
$$
Observe that for any free filter $\mathcal F$ on $X$ we have $\{0,1\}\subseteq\bigcap_{F\in\F}FF$. By Proposition~\ref{t:iT1}, the Taimanov semigroup is injectively $\mathsf{T_{\!1}S}$-closed. We claim that for the ideal $J=\{0,1\}\subset X$, the quotient semigroup $X/J$ is not $\mathsf{T_{\!z}S}$-closed. Observe that $X/J$ is a semigroup with trivial multiplication, i.e., $ab=J\in X/J$ for any $a,b\in X/J$. Take any Hausdorff zero-dimensional space $Y$ containing $X/J$ as a non-closed discrete subspace. Endow $Y$ with the continuous semigroup operation defined by $xy=J\in X/J\subset Y$ for all $x,y\in Y$. Since $X/J$ is a non-closed discrete subsemigroup of zero-dimensional topological semigroup $Y$,  the semigroup $X$ is not ideally $\mathsf{T_{\!z}S}$-closed.
\end{example}

The following trivial characterization of $\mathsf{T_{\!0}S}$-closed semigroups explains why we restrict our attention to  topological semigroups satisfying the separation axioms $T_i$ for $i\ge 1$.

\begin{proposition} For a topological semigroup $X\in\mathsf{T_{\!0}S}$ the following conditions are equivalent:
\begin{enumerate}
\item $X$ is $\mathsf{T_{\!0}S}$-closed;
\item $X$ is absolutely $\mathsf{T_{\!0}S}$-closed;
\item $X=\emptyset$.
\end{enumerate}
\end{proposition}

\begin{proof} The implications $(3)\Ra(2)\Ra(1)$ are trivial. To see that $(1)\Ra(3)$, assume that the semigroup $X$ is not empty.  On the $0$-extension $X^0$ of $X$ consider the topology $\tau^0=\{X^0\}\cup\tau$ where $\tau$ is the topology of $X$. It is easy to see that $(X^0,\tau^0)$ is a $T_0$ topological semigroup containing $X$ as a non-closed subsemigroup. Consequently, $X$ is not $\mathsf{T_{\!0}S}$-closed.
\end{proof}

\section{$\C$-closedness and $\C$-discreteness}

In this section we outline a connection between (injectively) $\C$-closed  and $\C$-discrete  semigroups.

\begin{proposition}\label{discr}
Every $\mathsf{T_{\!1}S}$-closed topological semigroup $X\in \mathsf{T_{\!1}S}$ is discrete.
\end{proposition}

\begin{proof}
To derive a contradiction, assume that the topological semigroup $X$ contains a non-isolated point $a$. Take any point $b\notin X$ and consider the set $Y=X\cup\{b\}$ endowed with the semigroup operation defined as follows:
\begin{itemize}
\item $X$ is a subsemigroup of $Y$;
\item $bb=aa$;
\item for every $x\in X$, $bx=ax$ and $xb=xa$.
\end{itemize}
By $\tau_X$ we denote the topology of $X$. Let $\tau_Y$ be the topology on $Y$ generated by the base $\tau_X\cup\{(U\setminus\{a\})\cup\{b\}:a\in U\in\tau_X\}$. It is straightforward to check that $(Y,\tau_Y)$ is a $T_1$ topological semigroup containing $X$ as a non-closed subsemigroup, which contradicts the $\mathsf{T_{\!1}S}$-closedness of $X$.
\end{proof}

\begin{proposition}\label{p:iT1}
For a semigroup $X$ the following conditions are equivalent:
\begin{enumerate}
\item $X$ is injectively $\mathsf{T_{\!1}S}$-closed;
\item $X$ is $\mathsf{T_{\!1}S}$-closed and $\mathsf{T_{\!1}S}$-discrete.
\end{enumerate}
\end{proposition}

\begin{proof}
$(1)\Ra(2)$: Assume that the semigroup $X$ is injectively $\mathsf{T_{\!1}S}$-closed. Then $X$ is $\mathsf{T_{\!1}S}$-closed.  To see that $X$ is $\mathsf{T_{\!1}S}$-discrete, consider any injective homomorphism $i:X\to Y$ to a $T_1$ topological semigroup $Y$. The injective $\mathsf{T_{\!1}S}$-closedness of $X$ implies the (injective) $\mathsf{T_{\!1}S}$-closedness of the topological semigroup $i[X]$. By Proposition~\ref{discr}, the topology of $i[X]$ is discrete, which means that the semigroup $X$ is $\mathsf{T_{\!1}S}$-discrete.
\smallskip

$(2)\Ra(1)$:
Assume that  a semigroup $X$ is $\mathsf{T_{\!1}S}$-closed and $\mathsf{T_{\!1}S}$-discrete. Consider any injective homomorphism $i:X\to Y$ into a $T_1$ topological semigroup $Y$. Since $X$ is $\mathsf{T_{\!1}S}$-discrete, the topological space $i[X]$ is discrete and the map $i:X\to Y$ is a topological embedding of $X$ endowed with the discrete topology into $Y$. Since the discrete topological semigroup $X$ is $\mathsf{T_{\!1}S}$-closed, the image $i[X]$ is closed in $Y$.
\end{proof}

\begin{proposition}\label{p:aT1S} For a  semigroup $X$ the following conditions are equivalent:
\begin{enumerate}
\item $X$ is absolutely $\mathsf{T_{\!1}S}$-closed;
\item $X$ is projectively $\mathsf{T_{\!1}S}$-closed and projectively $\mathsf{T_{\!1}S}$-discrete.
\end{enumerate}
\end{proposition}

\begin{proof}
Let $X$ be an absolutely $\mathsf{T_{\!1}S}$-closed semigroup.
Then $X$ is projectively $\mathsf{T_{\!1}S}$-closed and every homomorphic image of $X$ is injectively $\mathsf{T_{\!1}S}$-closed. Proposition~\ref{p:iT1} implies that $X$ is projectively $\mathsf{T_{\!1}S}$-discrete.
\smallskip

Assume that a  semigroup $X$ is projectively $\mathsf{T_{\!1}S}$-closed and projectively $\mathsf{T_{\!1}S}$-discrete.
Fix any  homomorphism $h:X\rightarrow Y$ to a $T_1$ topological semigroup $Y$. Since $X$ is projectively $\mathsf{T_{\!1}S}$-discrete, the subspace $h[X]$ of $Y$ is discrete. Since $X$ is projectively $\mathsf{T_{\!1}S}$-closed, the image $h[X]$ is closed in $Y$, witnessing that $X$ is absolutely $\mathsf{T_{\!1}S}$-closed.
\end{proof}

The above propositions motivate the problem of recognizing $\C$-discrete (or else $\C$-nontopolo\-gizable) semigroups. This problem has been considered by many authors, see \cite{BPS, DI, KOO, Kotov, Malcev, Mar, Taimanov78, Tro, Ztop}. The $\mathsf{T_{\!i}S}$-discreteness of {\em countable} semigroups can be characterized with the help of suitable Zariski topologies, see  \cite{DS1, DS2, DT0, DT1, DT2, EJM}.

For a semigroup $X$ its
\begin{itemize}
\item {\em Zariski topology} $\Zeta_X$ is the topology on $X$ generated by the subbase consisting of the sets $\{x\in X:f(x)\ne b\}$ and $\{x\in X:f(x)\ne g(x)\}$ where $b\in X$ and $f,g$ are semigroup polynomials on $X$;
\item {\em Zariski $T_1$-topology}  $\Zeta'_X$ is the topology on $X$ generated by the subbase consisting of the sets $\{x\in X:f(x)\ne b\}$ where $b\in X$ and $f$ is a semigroup polynomial on $X$.
\end{itemize}

The following theorem was proved independently by Podewski~\cite{Podewski} and Taimanov~\cite{Taimanov78}.

\begin{theorem}[Podewski--Taimanov]\label{t:Podewski}A countable semigroup $X$ is $\mathsf{T_{\!1}S}$-discrete if and only if its Zariski $T_1$-topology $\Zeta'_X$ is discrete.
\end{theorem}

The next theorem was announced by Taimanov \cite{Taimanov78} and proved by Kotov~\cite{Kotov}.

\begin{theorem}[Taimanov--Kotov]
A countable semigroup $X$ is $\mathsf{T_{\!2}S}$-discrete if and only if $X$ is $\mathsf{T_{\!z}S}$-discrete if and only if the Zariski topology $\Zeta_X$ is discrete.
\end{theorem}

We finish this section with the following example of a projectively $\mathsf{T_{\!1}S}$-closed and absolutely $\mathsf{T_{\!2}S}$-closed semilattice which is not injectively $\mathsf{T_{\!1}S}$-closed. We recall that a {\em semilattice} is a commutative semigroup of idempotents.

\begin{example} Given an infinite cardinal $\kappa$, consider the semilattice $X_\kappa=(\kappa,*)$ endowed with the binary operation
$$x*y=
\begin{cases}
x,&\mbox{if $x= y$};\\
0,&\mbox{otherwise}.
\end{cases}
$$Observe that for every free filter $\F$ on $X_\kappa$ the filter $\F\F$ is neither principal nor free. Applying Proposition~\ref{t:T1}, we conclude that the semilattice $X$ is $\mathsf{T_{\!1}S}$-closed. Since any homomorphic image of $X_\kappa$ is isomorphic to the semilattice $X_\lambda$ for some nonzero cardinal $\lambda\le\kappa$, the semilattice $X_\kappa$ is projectively $\mathsf{T_{\!1}S}$-closed. By \cite{BBm}, the semilattice $X_\kappa$ is absolutely $\mathsf{T_{\!2}S}$-closed. Since $X_\kappa$ admits a non-discrete $T_1$ semigroup topology $\tau=\{U\subseteq \kappa:0\in U\Ra ( |\kappa\setminus U|<\w)\}$, the semigroup $X_\kappa$ is $\mathsf{T_{\!1}S}$-topologizable. Then Proposition~\ref{p:iT1} implies that $X$ is not injectively $\mathsf{T_{\!1}S}$-closed.
\end{example}

\section{Polybounded semigroups}\label{s:poly}

In this section we establish some structural properties of polybounded semigroups. We start with the following three characterizations of 
polyboundedness.

\begin{proposition}\label{p:poly1} For a nonempty semigroup $X$ the following conditions are equivalent:
\begin{enumerate}
\item $X$ is polybounded;
\item for every $a,b\in X$ the semigroup $aXb$ is polybounded;
\item for some $a,b\in X$ the semigroup $aXb$ is polybounded.
\end{enumerate}
\end{proposition}

\begin{proof} To prove that $(1)\Ra(2)$, assume that the semigroup $X$ is polybounded and hence $X=\bigcup_{i=1}^np_i^{-1}(c_i)$ for some elements $c_1,\dots,c_n\in X$ and semigroup polynomials $p_1,\dots,p_n$ on $X$. Given any elements $a,b\in X$, for every $i\in\{1,\dots,n\}$, consider the function $f_i:aXb\to aXb$ defined by $f_i(x)=ap_i(bbxaa)b$, and observe that $f_i$ is a semigroup polynomial on $aXb$.
We claim that $aXb=\bigcup_{i=1}^n f^{-1}_i(acb)$. Indeed, for every $x\in aXb$ there exists $i\in\{1,\dots,n\}$ such that $p_i(bbxaa)=c_i$ and hence $f_i(x)=ap_i(bbxaa)b=ac_ib$ and finally $x\in f_i^{-1}(ac_ib)$, witnessing that the semigroup $aXb$ is polybounded.
\smallskip

The implication $(2)\Ra(3)$ is trivial.
\smallskip

To prove that $(3)\Ra(1)$, assume that for some $a,b\in X$ the semigroup $aXb$ is polybounded. Then $aXb=\bigcup_{i=1}^np_i^{-1}(c_i)$ for some elements $c_1,\dots,c_n\in aXb$ and some semigroup polynomials $p_1,\dots,p_n$ on $aXb$. For every $i\in\{1,\dots,n\}$, consider the semigroup polynomial $f_i:X\to X$, $f_i:x\mapsto p_i(axb)$, and observe that $X=\bigcup_{i=1}^nf_i^{-1}(c_i)$, witnessing that the semigroup $X$ is polybounded.
\end{proof}

\begin{proposition}\label{p:poly2} A nonempty semigroup $X$ is polybounded if and only if  $X=\bigcup_{i=1}^np_i^{-1}(b)$ for some $b\in X$ and semigroup polynomials $p_1,\dots,p_n$ on $X$.
\end{proposition}

\begin{proof} The ``if'' part is trivial. To prove the ``only if'' part, assume that $X$ is polybounded and find elements $b_1,\dots,b_n\in X$ and semigroup polynomials $p_1,\dots,p_n$ on $X$ such that $X=\bigcup_{i=1}^n p_i^{-1}(b_i)$. We lose no generality assuming that $n>1$.

Let $b=b_1\cdots b_n$. For every $i\in\{1,\dots,n\}$, consider the semigroup polynomial $f_i:X\to X$ defined by
$$f_i(x)=\begin{cases}
p_1(x)b_2\cdots b_n&\mbox{if $i=1<n$};\\
b_1\cdots b_{n-1}p_n(x)&\mbox{if $1<i=n$};\\
b_1\cdots b_{i-1}p_i(x)b_{i+1}\cdots b_n&\mbox{if $1<i<n$}.
\end{cases}
$$It is easy to see that $p_i^{-1}(b_i)\subseteq f_i^{-1}(b)$ and hence $X=\bigcup_{i=1}^np_i^{-1}(b)\subseteq \bigcup_{i=1}^nf_i^{-1}(b)\subseteq X$.
\end{proof}

\begin{proposition}\label{p:poly3} A nonempty semigroup $X$ is polybounded if and only if the semigroup $K=\bigcap_{x\in X}X^1xX^1$ is nonempty and polybounded.
\end{proposition}

\begin{proof} To prove the ``only if'' part, assume that the semigroup $X$ is polybounded. By Proposition~\ref{p:poly2}, $X=\bigcup_{i=1}^np_i^{-1}(b)$ for some $b\in X$ and some semigroup polynomials $p_1,\dots,p_n$. We claim that $b$ belongs to the ideal $K=\bigcap_{x\in X}X^1xX^1$. Indeed, given any $x\in X$, we can find $i\in\{1,\dots,n\}$ such that $b=p_i(x)\in X^1xX^1$. Therefore, the semigroup $K$ contains the element $b$ and hence $K$ is not empty. In fact, $K$ is the smallest nonempty ideal in $X$. To see that the semigroup $K$ is polybounded, for every $i\in\{1,\dots,n\}$ consider the function $f_i:K\to K$, $f_i(x)=p_i(bxb)$, and observe that $f_i$ is a semigroup polynomial on $K$. By the choice of the polynomials $p_1,\dots,p_n$, for every $x\in K$ there exists $i\in\{1,\ldots,n\}$ such that $bxb\in p_i^{-1}(b)$. It follows that $x\in f_i^{-1}(b)$, witnessing that $K=\bigcup_{i=1}^nf_i^{-1}(b)$. Hence the semigroup $K$ is polybounded. 
\smallskip

To prove the ``if'' part, assume that the semigroup $K=\bigcap_{x\in X}X^1xX^1$ is nonempty and polybounded. By Proposition~\ref{p:poly2}, $K=\bigcup_{i=1}^np_i^{-1}(b)$ for some $b\in K$ and some semigroup polynomials $p_1,\dots,p_n$ on $K$. Since $K$ is an ideal in $X$, for every $x\in X$ the element $bxb$ belongs to $K$. Then for every $i\in\{1,\dots,n\}$, the function $f_i:X\to X$, $f_i:x\mapsto p_i(bxb)$, is a well-defined semigroup polynomial on $X$. The equality  $K=\bigcup_{i=1}^np_i^{-1}(b)$ implies the equality $X=\bigcup_{i=1}^nf_i^{-1}(b)$, witnessing that the semigroup $X$ is polybounded.
\end{proof}

Next, we show that the class of polybounded semigroups is closed under taking finite products and quotients.

\begin{proposition}\label{qd}
Quotients and finite products of polybounded semigroups are polybounded.
\end{proposition}

\begin{proof} Assume that $X$ is a polybounded semigroup, and hence $X=\bigcup_{p\in P}\bigcup_{b\in B}p^{-1}(b)$ for some finite set $B\subseteq X$ and some finite set $P$ of semigroup polynomials on $X$. For every polynomial $p\in P$ find a number $n_p\in \IN$ and elements $a_{p,0},a_{p,1},\dots,a_{p,n_p}\in X^1$ such that $p(x)=a_{p,0}xa_{p,1}x\cdots xa_{p,n_p}$ for all $x\in X$.

Let $Y=X/_\approx$ be a quotient semigroup of $X$, and  $q:X\to Y$ be the quotient homomorphism. For every $p\in P$ and $i\in\{0,\dots,n_p\}$, let $\tilde a_{p,i}=q(a_{p,i})$. Let $\tilde p$ be the semigroup polynomial on $Y$ defined by $\tilde p(y)=\tilde a_{p,0}y\tilde a_{p,1}y\cdots y\tilde a_{p,n_p}$. Let $\tilde P=\{\tilde p:p\in P\}$ and $\tilde B=\{\tilde b:b\in B\}$, where $\tilde b=q(b)$ for $b\in B$. Since $q$ is a homomorphism, for every $p\in P$, $b\in B$ and $x\in p^{-1}(b)$ we have $\tilde p(q(x))=q(p(x))=q(b)=\tilde b$. Then $X=\bigcup_{p\in P}\bigcup_{b\in B}p^{-1}(b)$ implies $Y=\bigcup_{\tilde p\in \tilde P}\bigcup_{\tilde b\in \tilde B}\tilde p^{-1}(\tilde b)$, which means that the quotient semigroup $Y=X/_\approx$ is polybounded.
\smallskip

To show that finite products of polybounded semigroups are polybounded, it suffices to prove that for any nonempty polybounded semigroups $X,Y$ their product $X\times Y$ is polybounded. By Proposition~\ref{p:poly2}, $X=\bigcup_{f\in P_X}f^{-1}(b_X)$ for some $b_X\in X$ and some finite set $P_X$ of semigroup polynomials on $X$.
By the same reason, $Y=\bigcup_{g\in P_Y}g^{-1}(b_Y)$ for some $b_Y\in Y$ and some finite set $P_Y$ of semigroup polynomials on $Y$.

For every polynomial $f\in P_X$, find $\deg(f)\in \IN$ and elements $a_{f,0},\dots,a_{f,\deg(f)}\in X^1$ such that $f(x)=a_{f,0}x a_{f,1}\cdots xa_{f,\deg(f)}$ for all $x\in X$. Also for every  polynomial $g\in P_Y$ find $\deg(g)\in \IN$ and elements $a_{g,0},a_{g,1}\dots,a_{g,\deg(g)}\in Y^1$ such that $g(y)=a_{g,0}y a_{g,1}\cdots ya_{g,\deg(g)}$ for all $y\in Y$.

For any semigroup polynomials $f\in P_X$ and $g\in P_Y$, consider the function $$p_{f,g}:X\times Y\to X\times Y,\quad p_{f,g}:(x,y)\mapsto (f(x)^{\deg(g)},g(y)^{\deg(f)}),$$ and observe that $p_{f,g}$ is a semigroup polynomial (of degree $\deg(f)\cdot \deg(g)$) on $X\times Y$.

Since $$X\times Y=\bigcup_{f\in P_X}\bigcup_{g\in P_Y}p_{f,g}^{-1}(b_X^{\deg(g)},b_Y^{\deg(f)}),$$
the semigroup $X\times Y$ is polybounded.
%
%
%
\end{proof}

The following example provides a simple method for constructing polybounded groups.

\begin{example}\label{ex} For any commutative group $X$ with neutral element $0$, consider the group $X\rtimes \{-1,1\}$ endowed with the group operation $$\langle x,i\rangle*\langle y,j\rangle=\langle xy^i,ij\rangle.$$ The group $X\rtimes\{-1,1\}$ is polybounded since
$$X\rtimes \{-1,1\}=\{x\in X\rtimes \{-1,1\}:ax^2ax^2=e\},$$
where $e=(0,1)$ and $a=(0,-1)$.
\end{example}

\begin{lemma}\label{new} A semigroup $X$ is polybounded, if some point of $X$ is isolated in the Zariski topology $\Zeta'_X$.
\end{lemma}

\begin{proof}
Let $a$ be an isolated point of the topological space $(X,\Zeta'_X)$.
By the definition of the topology $\Zeta'_X$, there exist semigroup polynomials $f_1,\ldots f_n$ on $X$ and elements $b_1,\ldots, b_n\in X$ such that $$\{a\}=X\setminus \bigcup_{i=1}^n\{x\in X: f_i(x)=b_i\}.$$Consider the semigroup polynomial $f_0:X\to X$, $f_0:x\mapsto x$, and let $b_0=a$. Then $X=\bigcup_{i=0}^n\{x\in X: f_i(x)=b_i\}$, witnessing that $X$ is polybounded.
\end{proof}

\begin{remark} It is straightforward to check that the group $G=\mathbb Z\rtimes \{-1,1\}$ (see Example~\ref{ex}) is polybounded, but the space $(G,\Zeta'_G)$ is homeomorphic to the topological sum of two countable spaces endowed with the cofinite topology, implying that the space $(G,\Zeta'_G)$ has no isolated points. Hence the implication of Lemma~\ref{new} cannot be reversed even for countable groups.
\end{remark}

\begin{remark} By \cite{BanS}, for every infinite cardinal $\kappa$ with $\kappa^+=2^\kappa$ there exists a non-polybounded absolutely $\mathsf{T_{\!1}S}$-closed simple group $G$ of cardinality $|G|=\kappa^+$. By Proposition~\ref{p:aT1S}, the absolutely $\mathsf{T_{\!1}S}$-closed group $G$ is projectively $\mathsf{T_{\!1}S}$-discrete. Since $G$ is not polybounded, Lemma~\ref{new} ensures that the Zariski topology $\Zeta'_X$ has no isolated points. This example shows that Podewski--Taimanov Theorem~\ref{t:Podewski} does not extend to uncountable (semi)groups.
\end{remark}

\section{Polybounded semigroups with finite-to-one shifts}\label{s:poly-shift}

In this section we establish some specific properties of polybounded semigroups with finite-to-one shifts. The principal tool here is the notion of a pruned polynomial.

A semigroup polynomial  $p:X\to X$, $p:x\mapsto a_0xa_1\cdots xa_n$, on a semigroup $X$ is said to be {\em pruned} if $a_0=1=a_n$. Obviously, for every semigroup polynomial $f$ on $X$ there exist{\color{green}\st{s}} a pruned semigroup polynomial $p$ on $X$ and elements $a,b\in X^1$ such that $f(x)=ap(x)b$ for every $x\in X$.

\begin{lemma}\label{l:pruned} For every polybounded semigroup $X$ with finite-to-one shifts there exist a finite set $B\subseteq X$ and a finite set $P$ of pruned semigroup polynomials on $X$ such that $X=\bigcup_{p\in P}\bigcup_{b\in B}p^{-1}(b)$.
\end{lemma}

\begin{proof} By the polyboundedness of $X$, there exist a finite set $A\subseteq X$ and a finite set $F$ of semigroup polynomials on $X$ such that $X=\bigcup_{f\in F}\bigcup_{a\in A}f^{-1}(a)$.  For every semigroup polynomial $f\in F$, find a pruned semigroup polynomial $p_f:X\to X$ and elements $c_f,d_f\in X^1$ such that $f(x)=c_fp_f(x)d_f$ for all $x\in X$. Since the semigroup $X$ has finite-to-one shifts, the set $B=\bigcup_{f\in F}\bigcup_{a\in A}\{x\in X:c_fxd_f=a\}$ is finite. Let $P=\{p_f:f\in F\}$. It is easy to see that $X=\bigcup_{p\in P}\bigcup_{b\in B}p^{-1}(b)$.
\end{proof}

An element $x$ of a semigroup $X$ is called {\em regular} if $x=xx^{-1}x$ for some element $x^{-1}\in X$. A semigroup $X$ is {\em regular} if every element of $X$ is regular.

\begin{lemma}\label{l:regular} For every polybounded semigroup $X$ with finite-to-one shifts, there exist a finite set $P$ of pruned semigroup polynomials on $X$ and a finite set $B$ of regular elements of $X$ such that $X=\bigcup_{p\in P}\bigcup_{b\in B}p^{-1}(b)$.
\end{lemma}

\begin{proof} By Lemma~\ref{l:pruned},  $X=\bigcup_{p\in P}\bigcup_{b\in B}p^{-1}(b)$ for some finite set $B\subseteq X$ and some finite set $P$ of pruned semigroup polynomials on $X$. We can assume that the cardinality of the set $B$ is the smallest possible.

Consider the semigroup polynomial $s:X\to X$, $s:x\mapsto xx$.
Since $BB\subseteq X=\bigcup_{p\in P}\bigcup_{b\in B}p^{-1}(b)$, for every $b\in B$ there exists $\varphi_b\in P$ such that $\varphi_b(b^2)\in B$. We claim that $\varphi_b(b^2)=b$ for every $b\in B$. Assuming that $\varphi_b(b^2)\ne b$ for some $b\in B$, we can consider the set $$P'=P\cup\{\varphi_b\circ s\circ \varphi_b\}$$of pruned semigroup polynomials on $X$. We claim that $X=\bigcup_{p\in P'}\bigcup_{c\in B\setminus\{b\}}p^{-1}(c)$. Take any element $x\in X$. If $p(x)=c$ for some $p\in P$ and $c\in B\setminus\{b\}$, then $x\in \bigcup_{p\in P'}\bigcup_{c\in B\setminus\{b\}}p^{-1}(c)$. Otherwise, $p(x)=b$ for each $p\in P$. In particular, $\varphi_b(x)=b$ and, consequently, $\varphi_b\circ s\circ \varphi_b(x)=\varphi_b(b^2)\in B\setminus\{b\}$. Therefore, $X=\bigcup_{p\in P'}\bigcup_{c\in B\setminus\{b\}}p^{-1}(c)$ which contradicts the minimality of $B$. This contradiction shows that $\varphi_b(b^2)=b$ for all $b\in B$. Since the polynomial $\varphi_b$ is pruned, there exist elements $a_1,\ldots, a_{n-1}\in X$ such that
$$b=\varphi_b(b^2)=b^2a_1b^2\cdots a_{n-1}b^2=bb^{-1}b$$for the element $b^{-1}=ba_1b^2\cdots a_{n-1}b\in X$, which means that the elements of the set $B$ are regular.
\end{proof}

For a semigroup $X$ we denote by $E(X)$ the set $\{x\in X:xx=x\}$ of idempotents of $X$. 

\begin{proposition}\label{p:poly-E} For any nonempty polybounded semigroup $X$ with finite-to-one shifts, the set $E(X)$ is finite and nonempty.
\end{proposition}

\begin{proof} By Lemma~\ref{l:regular}, $X=\bigcup_{p\in P}\bigcup_{b\in B}p^{-1}(b)$ for a finite nonempty set $P$ of pruned semigroup polynomials on $X$ and a finite nonempty set $B$ of regular elements of $X$. By the regularity, for every $b\in B$ there exists $b^{-1}\in X$ such that $b=bb^{-1}b$ and hence $bb^{-1}$ is an idempotent. Therefore, the set $E(X)\supseteq\{bb^{-1}:b\in B\}$ is not empty. Since the semigroup $X$ has finite-to-one shifts, the set $F=\bigcup_{b\in B}\{x\in X:xbb^{-1}=bb^{-1}\}$ is finite.  We claim that $E(X)\subseteq F$. Indeed, for every $e\in E(X)$ we can find a pruned polynomial $p\in P$ and an element $b\in B$ such that $b=p(e)=ece$ for some $c\in X^1$. Then $ebb^{-1}=e(ece)b^{-1}=eceb^{-1}=bb^{-1}$ and hence $e\in F$.
\end{proof}

\begin{question} 
Does {\color{blue}every} nonempty polybounded semigroup contain an idempotent?
\end{question}

Now let us recall some standard definitions related to minimal (left or right) ideals.

A nonempty subset $I$ of a semigroup $X$ is called a {\em left ideal} (resp. {\em right ideal}) if $XI\subseteq I$ (resp. $IX\subseteq I$). A left (resp. right) ideal $I$ is called {\em minimal} if $I=J$ for any left (resp. right) ideal $J\subseteq X$ with $J\subseteq I$.

An ideal $I$ in a semigroup $X$ is called  {\em the minimal ideal} if $I\ne \emptyset$ and $I\subseteq J$ for every nonempty ideal $J\subseteq X$. 
A semigroup $X$ has the minimal ideal if and only if the semigroup $K=\bigcap_{x\in X}X^1xX^1$ is not empty, in which case $K$ is the minimal ideal of $X$.

A semigroup $X$ is called
\begin{itemize}
\item {\em simple} if each nonempty ideal in $X$ coincides with $X$;
\item {\em completely simple} if $X$ is simple and contains a minimal left ideal and a minimal right ideal;
\item {\em completely regular} if $X$ is a union of subgroups.
\end{itemize}
By Theorem 3.3.2 of \cite{Howie}, a nonempty simple semigroup is completely simple if and only if it is completely regular.

Proposition~\ref{p:poly3} implies that for every polybounded semigroup $X$ its minimal ideal $K=\bigcap_{x\in X}X^1xX^1$ is a simple polybounded semigroup. Our next aim is to show that for a polybounded semigroup with finite-to-one shifts, its minimal ideal $K$ is completely simple.

\begin{lemma}\label{l:minI} Every nonempty polybounded semigroup $X$ with finite-to-one shifts contains a minimal left ideal and a minimal right ideal.
\end{lemma}

\begin{proof} By Lemma~\ref{l:regular}, $X=\bigcup_{p\in P}\bigcup_{b\in B}p^{-1}(b)$ for some finite set $P$ of pruned polynomials and some finite set $B$ of regular elements.

To derive a contradiction, assume that the semigroup $X$ has no minimal right ideals. Then we can inductively construct sequences $(b_n)_{n\in\w}\in B^\w$ and $(x_n)_{n\in\w}\in X^\w$ such that for every $n\in\w$ the following conditions are satisfied:
\begin{itemize}
\item $x_n\in b_n X^1$ and $x_nX^1\ne b_nX^1$;
\item $b_{n+1}\in x_nX^1$.
\end{itemize}
To start the inductive construction of the sequences $(b_n)_{n\in\w}$ and $(x_n)_{n\in\w}$, take any element $b_0\in X$. Since $X$ has no minimal right ideals, the right ideal $b_0X^1$ is not minimal and hence there exists an element  $x_0\in b_0X^1$ such that $x_0X^1\ne b_0X^1$. Assume that for some $n\in\w$ we have constructed finite sequences $(b_k)_{k\le n}$ and $(x_k)_{k\le n}$. By the choice of the sets $P$ and $B$, there exist an element $b_{n+1}\in B$ and a pruned polynomial $p\in P$ such that $b_{n+1}=p(x_n)\in x_nX^1$. This completes the inductive step. Since the set $B$ is finite, there are two numbers $n<m$ such that $b_n=b_m$.

Then $$b_mX^1\subseteq x_{m-1}X^1\subset b_{m-1}X^1\subseteq x_{m-2}X^1\subseteq\ldots\subseteq b_nX^1$$and hence $b_mX^1\ne b_nX^1=b_mX^1$, which is a contradiction showing that $X$ has a minimal right ideal.

By analogy we can prove that $X$ has a minimal left ideal.
\end{proof}

\begin{proposition} If a nonempty polybounded semigroup $X$ has finite-to-one shifts, then its minimal ideal $K=\bigcap_{x\in X}X^1xX^1$ is a completely simple semigroup, which is the union $K=\bigcup_{e\in E(K)}eKe$ of finitely many polybounded groups $eKe$, $e\in E(X)$.
\end{proposition}

\begin{proof} By Proposition~\ref{p:poly3}, the semigroup $K=\bigcap_{x\in X}X^1xX^1$ is nonempty and polybounded. It is is clear that $K$ is simple and has finite-to-one shifts.  By Lemma~\ref{l:minI}, the polybounded semigroup $K$ has a minimal left ideal and a minimal right ideal and hence the simple semigroup $K$ is completely simple. By Theorem 3.3.2 in \cite{Howie}, the completely simple semigroup $S$ is completely regular and hence $K=\bigcup_{e\in E(K)}eKe$ is the union of subgroups.  By Proposition~\ref{p:poly-E}, the set $E(K)$ is finite, and by Proposition~\ref{p:poly1}, for every idempotent $e\in E(K)$ the group $eKe$ is polybounded.
\end{proof}

For a semigroup $X$ let
$$Z(X)=\{z\in X:\forall x\in X\;(xz=zx)\}$$be the {\em center} of $X$. We recall that a semigroup $X$ is called {\em bounded} if there exists $n\in\IN$ such that the $n$th power $x^n$ of any element $x\in X$ is an idempotent. It is clear that a bounded semigroup with finitely many idempotents is polybounded. The converse implication holds for commutative semigroups with finite-to-one shifts.

\begin{proposition}\label{p:bound-semibound} Let $X$ be a semigroup.
If $X$ has finite-to-one shifts and $Z(X)$ is polybounded in $X$, then $Z(X)$ is bounded.
\end{proposition}

\begin{proof}  Assuming that $Z(X)$ is not bounded, we conclude that $Z(X)$ and $X$ are nonempty. By the polyboundedness of $Z(X)$ in $X$, for some $n\in\IN$ there exist  semigroup polynomials $f_1,\dots,f_n$ on $X$ and  elements $b_1,\dots,b_n\in X$ such that  $Z(X)\subseteq \bigcup_{i=1}^nf_i^{-1}(b_i)$. For every $i\leq n$ there exist a number $p_i\in\IN$ and an element $a_i\in X^1$ such that $f_i(x)=a_ix^{p_i}$ for all $x\in Z(X)$. Let $p=\max_{i\le n}p_i$. Since $X$ has finite-to-one shifts, the set $F=\bigcup_{i\leq n}\{x\in Z(X):a_ix=b_i\}$ is finite.
Since $Z(X)$ is not bounded, there exists an element $z\in Z(X)$ such that for any distinct numbers $i,j\in\{1,\dots,p(1+n|F|)\}$ the powers $z^i,z^j$ are distinct. Since $$Z(X)=\bigcup_{i=1}^n\{x\in Z(X):f_i(x)=b_i\}=\bigcup_{i=1}^n\{x\in Z(X):a_ix^{p_i}=b_i\},$$ for every $j\in \IN$ there exists $i_j\in\{1,\dots,n\}$ such that $a_{i_j}(z^j)^{p_{i_j}}=b_{i_j}$.
By the Pigeonhole principle, for some $i\in\{1,\dots,n\}$ the set $J_i=\{j\in\{1,\dots,1+n\cdot|F|\}:i =i_j\}$ has cardinality $|J_i|>|F|$. Then for every $j\in J_i$ we have $a_iz^{jp_i}=b_i$ and hence $z^{jp_i}\in F$. Since $|J_i|>|F|$, there are two numbers $j<j'$ in $J_i$ such that $z^{jp_i}=z^{j'p_i}$. Since $\max\{jp_{i},j'p_i\}\leq (1+n|F|)p$, we obtain a contradiction with the choice of $z$.
\end{proof}

We apply Proposition~\ref{p:bound-semibound} in the proofs of the following characterizations of polybounded commutative semigroups (with finite-to-one shifts).

\begin{proposition}\label{p:comb} A commutative semigroup $X$ with finite-to-one shifts is polybounded if and only if $X$ is bounded and the set $E(X)$ is finite.
\end{proposition}

\begin{proof} To prove the ``if'' part, assume that $X$ is bounded and has finitely many idempotents. By the boundedness of $X$, there exists $n\in\IN$ such that $X=\bigcup_{e\in E(X)}\{x\in X:x^n=e\}$, which means that $X$ is polybounded.

To prove the ``only if'' part, assume that $X$ is polybounded. Being commutative, the semigroup $X$ coincides with its center $Z(X)$. By Proposition~\ref{p:bound-semibound}, the semigroup $Z(X)=X$ is bounded and by Proposition~\ref{p:poly-E}, the set $E(X)$ is finite.
\end{proof}

\begin{proposition} A nonempty commutative semigroup $X$ is polybounded if and only if the semigroup $K=\bigcap_{x\in X}X^1xX^1$ is nonempty and bounded.
\end{proposition}

\begin{proof}
Assume that the {\color{blue}commutative} semigroup $X$ is {\color{green}\st{commutative and}} polybounded. Proposition~\ref{p:poly3} implies that the semigroup $K=\bigcap_{x\in X}X^1xX^1$ is nonempty and polybounded. Clearly, $K$ is the minimal ideal of $X$. Using the commutativity of $X$ it can be checked that $K=aK=Ka$ for any $a\in K$, {\color{blue}which implies that $K$ is a group, see \cite[p.\,5]{Clifford-Preston-1961}}. By Proposition~\ref{p:comb}, {\color{blue}the group} $K$ is bounded.


Now assume that the semigroup $K=\bigcap_{x\in X^1}X^1xX^1$ is bounded and nonempty. The above arguments imply that $K$ is a group. Proposition~\ref{p:comb} ensures that {\color{blue}group} $K$ is polybounded. 
By Proposition~\ref{p:poly3}, the semigroup $X$ is polybounded.
\end{proof}

The following simple example shows that bounded commutative semigroups are not necessarily polybounded.

\begin{example} The commutative semigroup $\mathbb \w$ with the operation of maximum has finite-to-one shifts and is bounded but not polybounded.
\end{example}

\begin{remark}
By Proposition~\ref{p:comb}, the infinite product of finite cyclic groups $\prod_{n\in\mathbb N}(\IZ/n\IZ)$ is not polybounded.
Also, Example~\ref{ex} shows that the polyboundedness is not inherited by taking subgroups,  whereas the boundedness is a hereditary property.
\end{remark}

\section{$\C$-closedness and polyboundedness}

In this section we study a relationship between polybounded and $\C$-closed semigroups. 

\begin{lemma}\label{t:cancel1} Let $X$ be a semigroup with finite-to-one shifts and $A$ be a polybounded subset in $X$. Then $A$ is closed in any  $T_1$ topological semigroup $Y$ that contains $X$ as a discrete subsemigroup.
\end{lemma}

\begin{proof}
We lose no generality assuming that the semigroup $X$ is nonempty. Being polybounded in $X$, the set $A$ is contained in the union $\bigcup_{i=1}^m f_i^{-1}(b_i)$ for some semigroup polynomials $f_1,\dots,f_m$ on $X$ and some elements $b_1,\dots,b_m\in X$. Each semigroup polynomial $f_i$ is of the form $f_i(x)=a_{i,0}x\dots x a_{i,n_i}$ for some $n_i\in\IN$ and $a_{i,0},\dots, a_{i,n_i}\in X^1$. Let $\bar f_i:Y\to Y$ be the semigroup polynomial on $Y$, defined by $\bar f_i(y)=a_{i,0}y\cdots ya_{i,n_i}$ for $y\in Y$. Since the space $Y$ is $T_1$ and the polynomials $\bar f_i$, $i\leq m$, are continuous, the subset $\bigcup_{i=1}^m\bar f_i^{-1}(b_i)$ of $Y$ is closed and hence contains the closure of $A$ in $Y$.
Assuming that the set $A$ is not closed in $Y$,
take any element $y\in\overline{A}\setminus A$ and find $i\in\{1,\dots,m\}$ such that $\bar f_i(y)=b_i$. Taking into account that the subspace $X$ of $Y$ is discrete, we conclude that $y\in \overline A\setminus A\subseteq Y\setminus X$. By the discreteness of $X$, there exists an open set $V\subseteq Y$ such that $V\cap X=\{b_i\}$. By the continuity of the semigroup operation on $Y$, there exists an open neighborhood $U\subseteq Y$ of $y$ such that $a_{i,0}Ua_{i,1}\cdots Ua_{i,n_i}\subseteq V$. Since $y\in \overline{X}\setminus X$, the set $U\cap X$ is infinite. Fix any element $u\in U\cap X$ and observe that the set
$$\{x\in X:a_{i,0}x(a_{i,1}u\cdots ua_{i,n_i})=b_i\}\supseteq U\cap X$$is infinite. But this is impossible, because $X$ has finite-to-one shifts.
\end{proof}

Lemma~\ref{t:cancel1} implies the following:

\begin{corollary}\label{t:cancel}
Each polybounded semigroup with finite-to-one shifts is $\mathsf{T_{\!1}S}$-closed.
\end{corollary}

\begin{remark} For every semigroup $S$ its $0$-extension $S^0$ is polybounded since $S^0=\{x\in S^0:0x=0\}$.  This trivial example shows that the finite-to-one shift property is essential in Lemma~\ref{t:cancel1} and Corollary~\ref{t:cancel}.
\end{remark}

\begin{theorem}\label{t:semibound}
Every zero-closed countable semigroup is polybounded.
\end{theorem}

\begin{proof}
Assume that a countable semigroup $X$ is not polybounded. Let $(b_n)_{n\in\IN}$ be an enumeration of elements of $X$.
We shall construct inductively a sequence $A=\{x_n\}_{n\in\omega}\subseteq X$ such that for each $n\in\omega$ the point $x_n$ satisfies the following condition:
\begin{itemize}
\item[$(*_n)$] for every $k\le n$ and elements $a_0,\dots,a_{k},a_{k+1}\in(\{1\}\cup\{b_i\}_{1\le i\le n}\cup\{x_i\}_{i<n})^{n}\subseteq X^1$ we have $f(x_n)\ne a_{k+1}$ where $f(x)=a_0xa_1\cdots xa_k$.
\end{itemize}
To start the inductive construction, choose any point $x_0\in X$ and observe that the condition $(*_0)$ is satisfied,  as $1\notin X$. 
Assume that for some $n\in\w$ we have already chosen pairwise distinct elements $x_0,\dots,x_n$ such that for each $i\leq n$ the condition $(*_i)$ is satisfied. The finiteness of the set $F_{n+1}=\big(\{1\}\cup\{b_k:1\le k\le n+1\}\cup\{x_k:k\leq n\}\big)^{n+1}$ implies that there exist only finitely many semigroup polynomials of the form $f(x)=a_0xa_1\dots xa_k$, where $k\leq n+1$ and $\{a_0,\ldots, a_{k}\}\subseteq F_{n+1}$.
Since the semigroup $X$ is not polybounded, there
exists an element $x_{n+1}\in X\setminus\{x_i:i\leq n\}$ which satisfies the condition $(*_{n+1})$. After completing the inductive construction we obtain
the desired set $A=\{x_n\}_{n\in\omega}$.



Consider the countable family
$$\mathcal K=\bigcup_{n=1}^\infty\{a_0Aa_1\cdots Aa_n:a_0,\dots,a_n\in X^1\}$$of subsets of $X$. Since
 $$(a_0Aa_1\cdots Aa_n)\cdot(b_0Ab_1\cdots Ab_m)=a_0Aa_1\cdots A(a_nb_0)Ab_1\cdots Ab_m$$ and $$b(a_0Aa_1\cdots Aa_n)c=(ba_0)Aa_1\cdots A(a_nc),$$the family $\mathcal K$ satisfies conditions (1) and (2) of Lemma~\ref{l0}.
\smallskip

To show that $\mathcal K$ satisfies condition (3) of Lemma~\ref{l0}, fix any set $K\in\mathcal K$ and elements $a,b,c\in X^1$. Find $n\in\IN$ and elements $a_0,\dots,a_n\in X^1$ such that $K=a_0Aa_1\cdots Aa_n$. Next, find $m\ge 2n$ such that $$\{aa_0,a_nb,c\}\cup\{a_0,a_1,\dots,a_n\}\subseteq\{b_0,\dots,b_m\}.$$ We claim that $\{z\in K:azb=c\}\subseteq\{a_0x_{i_1}a_1\cdots x_{i_n}a_n:i_1,\dots,i_n<m\}$. In the opposite case there exists an element $z\in K$ such that $azb=c$ and $z=a_0x_{j_1}a_1\cdots x_{j_n}a_n$ for some numbers $j_1,\dots,j_n\in\w$ with $j:=\max\{j_1,\dots,j_n\}\ge m$. Let $P=\{p\leq n: j_p=j\}$ and write $P$ as $P=\{p(1),\dots,p(t)\}$ for some numbers $p(1)<\ldots<p(t)$. Let
$$a_0'=aa_0x_{j_1}a_1\cdots x_{j_{p(1)-1}}a_{p(1)-1},\qquad a_t'=a_{p(t)}x_{j_{p(t)+1}}a_{p(t)+1}\cdots x_{j_n}a_{n}b,$$and for every $0<i<t$ put
$$a_i'=a_{p(i)}x_{j_{p(i)+1}}a_{p(i)+1}\cdots x_{j_{p(i+1)-1}}a_{p(i+1)-1}.$$

It follows that $$azb=aa_0x_{j_1}a_1\cdots x_{j_n}a_nb=f(x_j)$$for the semigroup polynomial $f(x)=a'_0xa_1'\cdots xa_t'$. Observe that
$$\{c,a_0',\dots,a'_t\}\subset (\{1\}\cup\{aa_0,a_nb,c\}\cup\{a_i\}_{i\le n}\cup\{x_i\}_{i<j})^{2n-1}\subseteq(\{1\}\cup\{b_i\}_{1\leq i\le j}\cup\{x_i\}_{i<j})^j.$$
Then $c=azb=f(x_j)$ which is not possible, since $x_j$ satisfies the condition $(*_j)$. This contradiction shows that the set $\{z\in K:azb=c\}\subseteq\{a_0x_{i_1}a_1\cdots x_{i_n}a_n:i_1,\dots,i_n< m\}$ is finite and hence the family $\mathcal K$ satisfies condition (3) of Lemma~\ref{l0}.
\smallskip

To show that $\mathcal K$ satisfies the condition (4) of Lemma~\ref{l0}, fix any $c\in X$ and sets $K,L\in\mathcal K$. Find elements $a'_0,\dots,a'_k,a''_0,\dots,a''_l\in X^1$ such that $K=a'_0Aa_1'\cdots Aa_k'$ and $L=a''_0Aa_1''\cdots Aa''_l$. Find a number $m\ge 2(k+l)$ such that $$\{a'_ka''_0,c\}\cup\{a'_0,\dots,a'_k\}\cup\{a''_0,\dots,a''_l\}\subseteq\{b_0,\dots,b_m\}.$$ Repeating the argument of the proof of condition (3), we can show that the set $$\{(x,y)\in K\times L:xy=c\}\subseteq\{(a'_0x_{i_1}\cdots x_{i_k}a'_k,a''_0x_{j_1}\cdots x_{j_l}a''_l):\max\{i_1,\dots,i_k,j_1,\dots,j_l\}<m\}$$  is finite.

Applying Lemma~\ref{l0}, we conclude that the semigroup $X$ is not zero-closed.
\end{proof}

\begin{remark} By \cite{BanS}, for every infinite cardinal $\kappa$ with $\kappa^+=2^\kappa$ there exists a non-polybounded absolutely $\mathsf{T_{\!1}S}$-closed group of cardinality $\kappa^+$. This  shows that Theorem~\ref{t:semibound} does not generalize to semigroups of arbitrary cardinality. On the other hand,  according to \cite{ZeroPoly}, the set-theoretic assumption $\mathrm{cov}(\mathcal M)=\mathfrak c$ implies that a zero-closed semigroup $X$ is polybounded if $X$ admits a compact Hausdorff semigroup topology or $X$ has a separable complete subinvariant metric. Here $\mathrm{cov}(\mathcal M)$ is the smallest cardinality of a cover of the real line by nowhere dense subsets. The Baire Theorem implies that $\w_1\le\mathrm{cov}(\mathcal M)\le\mathfrak c$. By Theorem 7.13 in \cite{Blass}, the equality  $\mathrm{cov}(\mathcal M)=\mathfrak c$ is equivalent to  Martin's Axiom for countable posets.
\end{remark}

\section{Proofs of the main results}\label{s:main}

{\em Proof of Theorem~\ref{t:group}:} Given a nonempty polybounded cancellative semigroup $X$, we should prove that $X$ is a group. By Lemma~\ref{l:regular}, $X=\bigcup_{p\in P}\bigcup_{b\in B}p^{-1}(b)$ for some  finite set $P$ of pruned semigroup polynomials on $X$ and some finite set $B$ of regular elements of $X$. Then for every $b\in B$ we can find an element $b^{-1}\in X$ such that $bb^{-1}b=b$. It follows that the elements $bb^{-1}$ and $b^{-1}b$ are idempotents.

By the cancellativity, $X$ contains a unique idempotent. Indeed, for any idempotents $e,f$ in $X$, the equalities $e(ef)=(ee)f=ef=e(ff)=(ef)f$ imply $f=ef=e$. Now we see that $X$ contains a unique idempotent $e$ and every element $b\in B$ is invertible in the sense that $bb^{-1}=e=b^{-1}b$ for some element $b^{-1}\in X$. By the cancellativity, for every $x\in X$ the equalities $xe=x(ee)=(xe)e$  and $ex=(ee)x=e(ex)$ imply $xe=x=ex$, which means that $e$ is the unit of the semigroup $X$.

For every $x\in X$ we can find $p\in P$ and $b\in B$ such that $p(x)=b$. Since the semigroup polynomial $p$ is pruned, $p(x)=xy$ for some $y\in X$. Then $$x(yb^{-1})x=(xy)b^{-1}x=p(x)b^{-1}x=bb^{-1}x=ex=x,$$ which means that the semigroup $X$ is regular. By \cite[Ex.11]{Howie}, the regular cancellative semigroup $X$ is a group.
\smallskip

{\em Proof of Theorem~\ref{char}:} Given a  countable semigroup $X$ with finite-to-one shifts, we should prove the equivalence of the following conditions:
\begin{enumerate}
\item $X$ is $\mathsf{T_{\!1}S}$-closed;
\item $X$ is $\mathsf{T_{\!z}S}$-closed;
\item $X$ is zero-closed;
\item $X$ is polybounded.
\end{enumerate}
The implications $(1)\Ra (2)\Ra (3)$ are trivial. The implications $(3)\Ra (4)$ and $(4)\Ra (1)$ follows from Theorem~\ref{t:semibound} and Corollary~\ref{t:cancel}, respectively. 
\smallskip

{\em Proof of Theorem~\ref{t:C=pC}:} Let $\mathsf i\in\{1,2,3,3\frac12,\mathsf z\}$ and $X$ be a semigroup whose all homomorphic images have finite-to-one shifts. Assuming that $X$ is $\mathsf{T_{\!i}S}$-closed, we will prove that $X$ is projectively $\mathsf{T_{\!i}S}$-closed. Let $h:X\to Y$ be any homomorphism to a topological semigroup $Y\in\mathsf{T_{\!i}S}$ such that $h[X]$ is a discrete subsemigroup of $Y$. We should prove that $h[X]$ is closed in $Y$. Since $\overline{h[X]}$ is a topological semigroup in the class $\mathsf{T_{\!\mathsf i}S}$, we lose no generality assuming that $h[X]$ is dense in $Y$.

 By our assumption, the semigroup $h[X]$ has finite-to-one shifts. We claim that $I=Y\setminus h[X]$ is an ideal in $Y= \overline{h[X]}$. Assuming the opposite, we can find $x\in I$ and $y\in Y$ such that $xy\notin I$ or $yx\notin I$. First we consider the case $xy\notin I$. It follows from $xy\notin I$ that $xy\in h[X]$. Since $h[X]$ is a discrete subspace of $Y$, there exists an open neighborhood $O_{x,y}\subseteq Y$ of $xy$ in $Y$ such that $O_{xy}\cap h[X]=\{xy\}$. By the continuity of the semigroup operation on $Y$, there exist neighborhoods $O_x$ and $O_y$ of the points $x$ and $y$ in $Y$ such that $O_x\cdot O_y\subseteq O_{xy}$. Choose any $a\in O_x\cap h[X]$ and observe that $a(O_y\cap h[X])\subseteq O_{xy}\cap h[X]=\{xy\}$. Since the set $O_y\cap h[X]$ is infinite, the semigroup $h[X]$ fails to have finite-to-one shifts. This contradiction shows that $xy\in I$. By analogy we can show that $yx\in I$. So, $I$ is an ideal in $Y$.

 Consider the semigroup $U_h(X,Y):= X\cup(Y\setminus h[X])$ that contains $X$ and $Y\setminus h[X]$ as subsemigroups and such that for any $x\in X$ and $y\in Y\setminus h[X]$, the products $xy$ and $yx$ in $U_h(X,Y)$ are defined by $xy=h(x)y$ and $yx=yh(x)$ in $Y\setminus h[X]$. The semigroup $U_h(X,Y)$ is endowed with the topology $\tau$ consisting of the sets $W\subseteq U_h(X,Y)$ such that  for any $y\in Y\setminus h[X]$ there exists a neighborhood $V_y$ of $y$ in $Y$ such that $(V_y\setminus h[X])\cup h^{-1}[V_y]\subseteq W$. By \cite[Theorem 18]{Ban}, $U_h(X,Y)$ is a topological semigroup in the class $\mathsf{T_{\!i}S}$ and $X$ is a discrete subsemigroup of $U_h(X,Y)$. Since the semigroup $X$ is $\mathsf{T_{\!i}S}$-closed, the set $X$ is closed in $U_h(X,Y)$ and consequently, $h[X]$ is closed in $Y$, witnessing that the semigroup $h[X]$ is projectively $\mathsf{T_{\!i}S}$-closed.

\smallskip

{\em Proof of Theorem~\ref{main2}:} We should prove that a countable semigroup $X$ with finite-to-one shifts is injectively $\mathsf{T_{\!1}S}$-closed if and only if $X$ is $\mathsf{T_{\!1}S}$-discrete. The ``only if'' part of this characterization follows from Proposition~\ref{p:iT1} (and requires no assumption on the semigroup $X$). To prove the ``if'' part,  assume that the semigroup $X$ is  $\mathsf{T_{\!1}S}$-discrete. By Podewski--Taimanov  Theorem~\ref{t:Podewski}, the space $(X,\Zeta'_X)$ is discrete, and by Lemma~\ref{new}, the semigroup $X$ is polybounded.  Corollary~\ref{t:cancel} implies that the polybounded semigroup $X$ is $\mathsf{T_{\!1}S}$-closed and Proposition~\ref{p:iT1} implies that $X$ is injectively $\mathsf{T_{\!1}S}$-closed.
\smallskip

{\em Proof of Theorem~\ref{tag}:} We should prove that a countable cancellative semigroup $X$ is absolutely $\mathsf{T_{\!1}S}$-closed if and only if $X$ is projectively $\mathsf{T_{\!1}S}$-discrete. The ``only if'' part of this characterization follows from Proposition~\ref{p:aT1S} (and requires no assumptions on the semigroup $X$). To prove the ``if'' part,  assume that the semigroup $X$ is  projectively $\mathsf{T_{\!1}S}$-discrete. We lose no generality assuming that the semigroup $X$ is nonempty. By Podewski--Taimanov Theorem~\ref{t:Podewski}, the space $(X,\Zeta'_X)$ is discrete, and by Lemma~\ref{new}, the semigroup $X$ is polybounded.
By Theorem~\ref{t:group}, the nonempty polybounded cancellative semigroup $X$ is a group. To prove that the group $X$ is absolutely $\mathsf{T_{\!1}S}$-closed, take any homomorphism $h:X\rightarrow Y$ into a $T_1$ topological semigroup $Y$. The projective $\mathsf{T_{\!1}S}$-discreteness of the group $X$ implies the $\mathsf{T_{\!1}S}$-discreteness of the group $h[X]\subseteq Y$. By Podewski--Taimanov Theorem~\ref{t:Podewski}, the space $(h[X],\Zeta'_{h[X]})$ is discrete, and by Lemma~\ref{new}, the group $h[X]$ is polybounded. Corollary~\ref{t:cancel} ensures that the group $h[X]$ is closed in $Y$.
\smallskip

{\em Proof of Theorem~\ref{t:aC-cgrp}:} Given a commutative cancellative semigroup $X$, and a class $\C$ of topological semigroups with $\mathsf{T_{\!z}S}\cap\mathsf{TG}\subseteq\C\subseteq\mathsf{T_{\!1}S}$, we should prove the equivalence of the following conditions:
\begin{enumerate}
\item $X$ is absolutely $\C$-closed;
\item $X$ is injectively $\C$-closed;
\item $X$ is $\C$-discrete;
\item $X$ is finite.
\end{enumerate}

Since (4) trivially implies  (1)--(3), and $(1)\Ra(2)$, it suffices to prove that the negation of (4) implies the negations of the conditions (2) and (3). So, assume that $X$ is an infinite commutative cancellative semigroup.  By \cite[p.34]{Clifford-Preston-1961}, the cancellative commutative semigroup $X$ is a subsemigroup of a commutative group $G$ such that $G=X-X$. Separately we consider two cases.
\smallskip

First assume that the group $G$ is bounded. Then the bounded subsemigroup $X$ of $G$ is a subgroup of $G$ and hence $G=X-X=X$. By Baer--Pr\"ufer Theorem 17.2 \cite{Fuchs}, the bounded group $G$ is isomorphic to the direct sum $\oplus_{\alpha\in\kappa}G_\alpha$ of finite cyclic groups $G_\alpha$. Since $X$ is infinite, the indexing cardinal $\kappa$ is infinite, too. Taking into account that $\oplus_{\alpha\in\kappa}G_\alpha$ is a non-closed non-discrete subgroup of the zero-dimensional compact topological group $\prod_{\alpha\in\kappa}G_\alpha\in\mathsf{T_{\!z}S}\cap\mathsf{TG}\subseteq\C$, we conclude that the semigroup $X=G$ is not injectively $\C$-closed and not $\C$-discrete.
\smallskip

Next, assume that the group $G$ is unbounded. Then the semigroup $X$ is unbounded as well. Choose any countable divisible subgroup $D$ of the multiplicative group $\IT=\{z\in\IC:|z|=1\}$ such that $\{z\in\IC:\exists n\in\IN\;(z^n=1)\}\subseteq D$ and $D$ contains an element of infinite order.

Denote by $H$ the set of all homomorphisms from the group $G$ to the group $D$. By Baer Theorem \cite[21.1]{Fuchs}, the diagonal homomorphism $\delta:G\to D^H$,
$\delta:x\mapsto (\varphi(x))_{\varphi\in H}$, is
injective. Endow the group $D^H$ with the subspace topology inherited from the compact topological group $\IT^H$.

We claim that the space $\delta[X]$ has no isolated points. In the opposite case, we could find a point $x\in X$ and an open set $U$ in $\IT^H$ such that $\delta[X]\cap U=\{\delta(x)\}$. Find an open neighborhood $V$ of the identity $e$ in the compact topological group $\IT^H$ such that $\delta(x)VV^{-1}\subseteq U$. By the compactness of the topological group $\IT^H$, there exists a finite set $F\subseteq\IT^H$ such that $\IT^H=VF$. Since the semigroup $X$ is unbounded, there exists an element $b\in X$ such that $e\ne \delta(b^i)\ne \delta(b^j)$ for any positive numbers $i<j\le|F|+1$. Since $\{\delta(b^i):1\le i\le |F|+1\}\subseteq VF$, by the Pigeonhole Principle, there exist $f\in F$ and positive numbers $i<j\le|F|+1$ such that $\{\delta(b^i),\delta(b^j)\}\subseteq  Vf$. Then $\delta(b^{j-i})=\delta(b^j)\delta(b^i)^{-1}\in Vf(Vf)^{-1}=VV^{-1}$ and then $\delta(xb^{j-i})\in\delta[X]\cap\delta(x)VV^{-1}\subseteq\delta[X]\cap U=\{\delta(x)\}$ and hence $\delta(b^{j-i})=e$, which contradicts the choice of $b$. This contradiction shows that the subspace $\delta[X]$ of the zero-dimensional Hausdorff group $D^H$ has no isolated points and hence the semigroup $X$ is not $\C$-discrete.

To show that $X$ is not injectively $\C$-closed, we need  the following claim.

\begin{claim}\label{cl:hom-inf} There exists a homomorphism $h:X\to D$ whose image $h[X]$ is infinite.
\end{claim}

\begin{proof} If the semigroup $X$ has an element $x$ of infinite order, then we can apply Baer Theorem \cite[21.1]{Fuchs} on extension of homomorphisms into  divisible groups and find a homomorphism $h:G\to D$ such that $h(x)$ is the element of infinite order in $D$. In this case the set $h[X]\supseteq\{h(x^n):n\in\IN\}$ is  infinite.

If every element of the semigroup $X$ has finite order, then $X=G$ is an unbounded periodic group. In this case a homomorphism $h:G\to D$ with infinite image $h[G]=h[X]$ exists by Claim 5.10 from \cite{BB}.
\end{proof}

Let $\overline{\delta[X]}$ be the closure of the semigroup $\delta[X]$ in the compact topological group $\IT^H$. For every $h\in H$ let $\pr_h:\IT^H\to\IT$, $\pr_h:y\mapsto y(h)$, denote the $h$-coordinate projection. By Claim~\ref{cl:hom-inf}, there exists a homomorphism $\varphi\in H$ whose image $\varphi[X]$ is infinite and hence $\varphi[X]\subseteq D$ is a dense subsemigroup of the compact topological group $\IT$. Observe that $\varphi=\pr_\varphi\circ\delta$ and hence $\varphi[X]=\pr_\varphi[\delta[X]]\subseteq \pr_\varphi[\overline{\delta[X]}]\subseteq\overline{\pr_\varphi[\delta[X]]}=\overline{\varphi[X]}=\IT$. The compactness of the set $\pr_\varphi[\overline{\delta[X]}]$ and the density of $\varphi[X]$ in $\IT$ imply that $\pr_\varphi[\overline{\delta[X]}]=\IT$. So, we can choose an element $y\in\overline{\delta[X]}\subseteq\IT^H$ such that $\pr_\varphi(y)\in\IT\setminus D$ and hence $y\in\overline{\delta[X]}\setminus\delta[X]$.

For every $h\in H$, let $D_h$ be the countable subgroup of $\IT$, generated by the set $D\cup\{\pr_h(y)\}$. Consider the group $\Pi=\prod_{h\in H}D_h$ endowed with the topology inherited from the compact topological group $\IT^H$. It is easy to see that $\Pi$ is a zero-dimensional topological group containing $\delta[X]$ as a non-closed subsemigroup and witnessing that the semigroup $X$ is not injectively $\C$-closed.

\section{Application of the polyboundedness to paratopological groups}
In this section we  present some applications of polyboundedness to paratopological groups.


A {\em paratopological group} is a group endowed with a semigroup topology. A paratopological group $G$ is a {\em topological group} if the operation of taking the inverse $G\to G$, $x\mapsto x^{-1}$, is continuous. In this case the topology of $G$ is called a {\em group topology}.

The problem of automatic continuity of the inversion in paratopological groups goes back to the classical work of Ellis~\cite{Ellis}. This problem was investigated by many authors (see surveys~\cite{BR} and~\cite{Tka} for more information).  A typical result says that the continuity of the inversion in paratopological groups follows from a suitable property of compactness type. For example,  a paratopological group $G$ is topological if $G$ possesses one of the following properties: locally compact, sequentially compact, totally countably compact, regular feebly compact or quasiregular 2-pseudocompact~\cite{Tka}.

The next proposition shows that every polybounded $T_1$ paratopological group is topological.

\begin{proposition}\label{p:topg}  Let $G$ be a $T_1$ paratopological group and $H\subseteq G$ be a subgroup which is polybounded in $G$. Then $H$ is a topological group.
\end{proposition}

\begin{proof} Being polybounded, the subgroup $H$ is contained in the union $\bigcup_{i=1}^nf_i^{-1}(b_i)$ for some semigroup polynomials $f_1,\dots,f_n:G\to G$ and some elements $b_1,\dots,b_n\in G$. 
Observe that if $f_i(x)=a_0xa_1\cdots xa_n$ for some $a_0,\dots,a_n\in G$, then $f_i^{-1}(b_i)=\tilde f_i^{-1}(e)$ for the semigroup polynomial $\tilde f_i(x)=xa_1\cdots xa_nb_i^{-1}a_0$. So, we can assume that for every $i\le n$ there exists a semigroup polynomial or a constant self-map $g_i$ of $X$ such that $f_i(x)=xg_i(x)$ for all $x\in G$, and $b_i=e$. It follows from 	$$H\subseteq \bigcup_{i=1}^n f_i^{-1}(e)=\bigcup_{i=1}^n\{x\in G:xg_i(x)=e\}$$ that
$x^{-1}\in\{g_i(x)\}_{i\le n}$ for every $x\in H$.

To prove that $H$ is a topological group we should prove that for any open neighborhood $U\subseteq G$ of $e$ there exists a neighborhood $V\subseteq G$ of $e$ such that $(H\cap V)^{-1}\subseteq U$. Fix any neighborhood $U$ of $e$ in $G$.
Since $G$ satisfies the separation axiom $T_1$, we can replace $U$ by a smaller neighborhood and assume that for every $i\leq n$, $g_i(e)^{-1}\notin U$, whenever $g_i(e)\ne e$. By the continuity of the semigroup operation in $G$, there exists a neighborhood $W\subseteq U$ of $e$ such that $WW\subseteq U$.
By the continuity of the functions $g_1,\dots, g_n$, there exists a neighborhood $V\subseteq W$ of $e$ such that for any $i\in\{1,\dots,n\}$ we have $g_i[V]\subseteq Wg_i(e)$. We claim that  $(H\cap V)^{-1}\subseteq U$. Indeed, fix any $x\in H\cap  V$. By the choice of the functions $g_1,\dots,g_n$, there exists $i\in\{1,\dots,n\}$ such that $xg_i(x)=e$. Then
$$e=xg_i(x)\in Vg_i[V]\subseteq  VW g_i(e)\subseteq WW g_i(e)\subseteq Ug_i(e).$$
Recall that if $g_i(e)\neq e$, then, by the choice of $U$, $g_i(e)^{-1}\notin U$. Consequently, $e\notin Ug_i(e)$ which contradicts to the above inclusion. Hence
$g_i(e)=e$. Finally, $$x^{-1}=g_i(x)\in g_i[V]\subseteq Wg_i(e)=We=W\subseteq U.$$
\end{proof}

\begin{proposition}
If a Hausdorff topological semigroup $Y$ contains a dense polybounded subgroup $X$, then $Y$ is a topological group.
\end{proposition}

\begin{proof} Let $e$ be the identity of the group $X$. The density of $X$ in $Y$ and the Hausdorff property of $Y$ implies that the (closed) set $\{y\in Y:ye=y=ey\}$ coincides with $Y=\overline{X}$, which means that $e$ is the identity of the semigroup $Y$.

Since the group $X$ is polybounded, there are elements $b_1,\dots,b_n\in X$ and semigroup polynomials $f_1,\dots,f_n$ on $X$ such that $$X=\bigcup_{i=1}^n f_i^{-1}(b_i).$$ Using the same trick as in the proof of Proposition~\ref{p:topg}, we can assume that each polynomial $f_i(x)$ is of the form $xg_i(x)$ where $g_i$ is a semigroup polynomial or a constant self-map of $X$ and $b_i=e$ for every $i\leq n$. Then for every $x\in X$ there exists $i\in\{1,\dots,n\}$ such that $x g_i(x)=e$, which means that $g_i(x)=x^{-1}$ and $xg_i(x)=e=g_i(x)x$.

The function $g_i:X\to X$ is of the form $g_i(x)=a_{i,0}x\dots x a_{i,n_i}$ for some $n_i\in\w$ and $a_{i,0},\dots, a_{i,n_i}\in X$. Let $\bar g_i:Y\to Y$ be the continuous function on $Y$ defined by $\bar g_i(y)=a_{i,0}y\dots ya_{i,n_i}$ for $y\in Y$. It follows that the set $\bigcup_{i=1}^n\{y\in Y:y\bar g_i(y)=e=\bar g_i(y)y\}$ coincides with $Y$, because it is closed and contains a dense subset $X$ in $Y$. Consequently, the semigroup $Y$ is polybounded and for every $y\in Y$ there exists an element $z\in \{\bar g_i(y):i\leq n\}\subset Y$ such that $yz=e=zy$, which means that $Y$ is a group. By Proposition~\ref{p:topg}, $Y$ is a topological group.
\end{proof}

Conditions implying that a cancellative topological semigroup $S$ is a topological group is another widely studied topic in Topological Algebra (see,~\cite{BG, HH, RS}). The following corollary of Theorem~\ref{t:group} and Proposition~\ref{p:topg} contributes to this field.



\begin{corollary}
Every nonempty polybounded cancellative $T_1$ topological semigroup is a topological group.
\end{corollary}




\section*{Acknowledgements}
The authors express their sincere thanks to the referee for many valuable suggestion which helped the authors to improve essentially the final version of the paper. 

\end{document}